\documentclass[openany, amssymb, psamsfonts]{amsart}
\usepackage{amssymb}
\usepackage{mathrsfs,comment}
\usepackage{stmaryrd}
\usepackage{tikz-cd}
\usepackage[normalem]{ulem}
\usepackage{url}
\usepackage{mathtools}
\usepackage{tikz}
\usetikzlibrary{calc,arrows}
\usepackage[all,arc,2cell]{xy}
\UseAllTwocells
\usepackage{enumitem}
\usepackage{hyperref} 
\usepackage[utf8]{inputenc}
\usepackage[T1]{fontenc}
\hypersetup{%
	bookmarksnumbered=true,%
	bookmarks=true,%
	colorlinks=true,%
	linkcolor=blue,%
	citecolor=blue,%
	filecolor=blue,%
	menucolor=blue,%
	pagecolor=blue,%
	urlcolor=blue,%
	pdfnewwindow=true,%
	pdfstartview=FitBH}

\def\makeautorefname#1#2{\expandafter\def\csname#1autorefname\endcsname{#2}}
\def\equationautorefname~#1\null{(#1)\null}
\makeautorefname{footnote}{footnote}%
\makeautorefname{item}{item}%
\makeautorefname{figure}{Figure}%
\makeautorefname{table}{Table}%
\makeautorefname{part}{Part}%
\makeautorefname{appendix}{Appendix}%
\makeautorefname{chapter}{Chapter}%
\makeautorefname{section}{Section}%
\makeautorefname{subsection}{Section}%
\makeautorefname{subsubsection}{Section}%
\makeautorefname{theorem}{Theorem}%
\makeautorefname{thm}{Theorem}%
\makeautorefname{cor}{Corollary}%
\makeautorefname{lem}{Lemma}%
\makeautorefname{prop}{Proposition}%
\makeautorefname{pro}{Property}
\makeautorefname{conj}{Conjecture}%
\makeautorefname{defn}{Definition}%
\makeautorefname{notn}{Notation}
\makeautorefname{notns}{Notations}
\makeautorefname{rem}{Remark}%
\makeautorefname{quest}{Question}%
\makeautorefname{exmp}{Example}%
\makeautorefname{ax}{Axiom}%
\makeautorefname{claim}{Claim}%
\makeautorefname{ass}{Assumption}%
\makeautorefname{asss}{Assumptions}%
\makeautorefname{con}{Construction}%
\makeautorefname{prob}{Problem}%
\makeautorefname{warn}{Warning}%
\makeautorefname{obs}{Observation}%
\makeautorefname{conv}{Convention}%

\newtheorem{thm}{Theorem}[section]

\newtheorem{prop}{Proposition}[section]
\newtheorem{lem}{Lemma}[section]

\newtheorem*{thm*}{Theorem}

\theoremstyle{definition}
\newtheorem{defn}{Definition}[section]

\newtheorem{rem}{Remark}[section]

\makeatletter
\let\c@obs=\c@thm
\let\c@cor=\c@thm
\let\c@prop=\c@thm
\let\c@lem=\c@thm
\let\c@prob=\c@thm
\let\c@con=\c@thm
\let\c@conj=\c@thm
\let\c@defn=\c@thm
\let\c@notn=\c@thm
\let\c@notns=\c@thm
\let\c@exmp=\c@thm
\let\c@ax=\c@thm
\let\c@pro=\c@thm
\let\c@ass=\c@thm
\let\c@warn=\c@thm
\let\c@rem=\c@thm
\let\c@sch=\c@thm
\let\c@equation\c@thm
\numberwithin{equation}{section}
\makeatother

\DeclareMathOperator{\SL}{SL}

\DeclareMathOperator{\Sp}{Sp}
\DeclareMathOperator{\SO}{SO}
\DeclareMathOperator{\SU}{SU}

\DeclareMathOperator{\Hess}{Hess}
\DeclareMathOperator{\Hom}{Hom}

\DeclareMathOperator{\ad}{ad}

\DeclareMathOperator{\tr}{tr}

\DeclareMathOperator{\grad}{grad}
\DeclareMathOperator{\vol}{vol}

\DeclareMathOperator{\supp}{supp}

\DeclareMathOperator{\FA}{FA}

\DeclareMathOperator{\ff}{\mathfrak{f}}

\newcommand{\bs}{\backslash}

\newcommand{\fa}{\mathfrak{a}}

\newcommand{\rank}{\text{rank}}
\renewcommand{\setminus}{-} 
\makeatletter
\newtheorem*{rep@theorem}{\rep@title}
\newcommand{\newreptheorem}[2]{%
	\newenvironment{rep#1}[1]{%
		\def\rep@title{#2 \ref{##1}}%
		\begin{rep@theorem}}%
		{\end{rep@theorem}}}
\makeatother

\newreptheorem{theorem}{Theorem}
\newreptheorem{cor}{Corollary}

\title{Minimal submanifolds and waists of locally symmetric spaces}

\author{Mikołaj Frączyk}
\address{Faculty of Mathematics and Computer Science, Jagiellonian University, ul. Łojasiewicza 6, 30-348 Krak{\'o}w, Poland}
\email{mikolaj.fraczyk@uj.edu.pl}

\author{Ben Lowe}
\address{Department of Mathematics\\University of Chicago\\Chicago, IL 60637}
\email{loweb24@uchicago.edu}

\date{\today}

\begin{document}
	\raggedbottom
	\maketitle
	
	\begin{abstract} 
    
    We show that compact locally symmetric manifolds $M$ with universal cover the symmetric space $X$ for  $\SL(n,\mathbb{R})$ form a topological higher $d$-expander family for $d\leq  n/8$.  We prove the same statement for $\SL(n,\mathbb{R})$ replaced by a split simple non-compact real Lie group $G$ and for $d$ linear in the rank of $G$. We accomplish this by showing that minimal submanifolds of low codimension in such $M$ must have volume comparable to the volume of $M$. Our proof is based on a new monotonicity formula for minimal submanifolds of $X$, together with bounds on the decay of matrix coefficients for unitary representations of higher rank Lie groups. We also give the first locally symmetric example of power-law systolic freedom.  This paper partially supersedes \cite{fl24}.

	\end{abstract} 
	
	\tableofcontents
	\section{Introduction}
	
	Expander graphs have been widely useful in combinatorics,  number theory, group theory, and many other branches of mathematics and computer science \cite{lubotzkyEXP}. They have a natural Riemannian analogue in the form of manifolds with a uniform spectral gap for the Laplace-Beltrami operator. A family of compact Riemannian manifolds $M_n$ with a uniform spectral gap is called an \textit{expander family}, and has the following property: there is a uniform $\lambda>0$ so that for every $M_n$ and every continuous map $f:M_n \rightarrow \mathbb{R}$ with rectifiable fibers, there is some $p \in \mathbb{R}$ so that 
	\begin{equation}\label{eq-topexpansion}
		\vol (f^{-1}(p)) \geq \lambda \vol(M_n).  
	\end{equation}
	These types of estimates are called \emph{waist inequalities}. In a series of works \cite{gr03, gr09,gromov2010pt2} Gromov initiated a program to prove statements of the following kind.  Suppose we are given two simplicial complexes, $X$ and $Y$, where $X$ is ``complicated" and $Y$ is lower dimensional. Then, any map $f\colon X\to Y$ must have at least one ``complicated" fiber $f^{-1}(y), y\in Y$. Historically volume was the first measure of complexity considered and it is the one we consider in this paper, but it is just one of many possible choices. 

\begin{defn} \label{defn-expansion}Let $X,Y$ be  Riemannian manifolds (or simplicial complexes) of dimensions $n,d$ respectively. Let $f\colon X\to Y$ be a continuous map. The \emph{waist} of $f$ is 
 $w(f):=\sup_{y\in Y} \mathcal H^{n-d}(f^{-1}(y))$, where $\mathcal H^{n-d}$ stands for the $(n-d)$-dimensional Hausdorff measure.The \emph{$d$-waist} of $X$ is then defined as $w_d(X):=\inf_{f\colon X\to \mathbb R^d} w(f)$, where the infimum is taken over a suitable class of maps with $(n-d)$-rectifiable fibers.  In this paper, we take the above infimum over a class of generic maps that we call generalized piecewise linear (see Section \ref{sec-PL} for the definition.)  
\end{defn}

\begin{rem}
A condition on the maps in question in the form of a genericity or real-analyticity assumption is present in all waist inequalities we are aware of in the literature (e.g., \cite{gr03}, \cite{gr09}, \cite{gromov2010pt2}, \cite{memarian2012lower},\cite{freedmanwidths},\cite{bader2024}), but it may only be necessary to require the maps to have rectifiable fibers of the proper dimension.  

\end{rem}

The families of spaces $\{ X_i\}_{i\in I}$ satisfying $w_d(X_i)\geq c \vol(X_i)$ are called \emph{ topological $d$-expanders}.  In what follows we refer to such families as \emph{$d$-expander families} for short. Gromov asked for examples of natural families of bounded geometry CW-complexes or Riemannian manifolds that are ``higher expanders", that is, $d$-expanders for $d\geq 2$.  The topological expansion property subsequently motivated other definitions of higher expansion like the cosystolic or coboundary expanders \cite{LubotzkyHDX, GotlKaufKHDX}. Gromov's question was answered by Evra-Kaufman \cite{EK24}, who exhibited bounded degree families  of $d$-expander simplicial complexes. In the present work, we show that natural  bounded geometry $d$-expander families can already be found among the higher rank locally symmetric spaces. 

\begin{thm}
For any $d> 0$ there exists an $N_0 \in \mathbb{Z}_+$ such that the following holds. If $n>N_0$, then the family of all compact quotients of the symmetric space for $SL(n,\mathbb{R})$ forms a $d$-expander family.

More generally, let $X= G/K$ be the symmetric space for a split simple higher rank non-compact real Lie group $G$.  Then if the rank of $G$ is larger than $N_0$, the family of compact quotients of $X$ form a $d$-expander family.  
\end{thm}
 
We believe our work presents a range of new possibilities for applications of minimal surface techniques to Gromov's program for non-positively curved locally symmetric spaces more generally, as well as to low dimensional actions of lattices in semisimple Lie groups. The interaction between representation theory and minimal surfaces is at the heart of our approach\footnote{Recent work by Song on  minimal surfaces in high dimensional spheres \cite{song2024random} also used representation theory, although in a completely different way.}. Our work draws on the min-max theory of minimal surfaces (via a connection to waists), which has seen a revival of activity (see \cite{marques2013applications} for a survey.) We hope there can be more applications of minimal surface theory to the study of the geometry of locally symmetric spaces and higher expansion.  

\subsection{Waist inequalities.}

For each $n$, let $X = G/K$   be the symmetric space for a split simple noncompact real Lie group $G$, and let $\{M_k\}$ be an enumeration of the compact Riemannian manifolds with universal cover $X$.  All manifolds in this paper are understood to be complete (i.e., without boundary), unless otherwise specified.  
	
	\begin{thm} \label{mt-higherrank}

There is a constant $c>0$ so that the $\{M_k\}$ form a higher $d$-expander family for $d < c \left( \rank(G)\right) $.  In the case that $G=SL(n,\mathbb{R})$, the $M_k$ form a higher $d$-expander family for $d \leq \lfloor n / 8 \rfloor$.

	\end{thm}

Bader-Sauer conjecture that the statement of the above theorem remains true for $d \leq \rank (G)$ \cite{bader2025higher}.  Our next theorem\footnote{This was the main result of a previous version of this manuscript.} provides a natural class of $2$-expander manifolds in the rank one setting, answering  the Riemannian version of Gromov's question in the $d=2$ case. 
\begin{thm}\label{mt-waist} There exists a constant $c>0$ such that any compact octonionic hyperbolic manifold $M$ satisfies $w_2(M)\geq c \vol(M).$ That is, compact octonionic hyperbolic manifolds form a $2$-expander family. 
\end{thm}
\begin{rem}
It is well known that if the fundamental group of a Riemannian manifold $M$ has Kazhdan's property $(T)$, then the family of finite covers of $M$ is an expander family. In a recent work, Bader and Sauer \cite{bader2024} show that it is also a $2$-expander family, thus giving another answer to Gromov's question for $d=2$ in the Riemannian setting. Recent work by Zung constructed compact negatively curved 3-manifolds $M$ that are almost 2-expanders: $w_2(M)\gtrapprox \vol(M)^{1-\epsilon}$ for every $\epsilon>0$ \cite{zung2024expansion}.

\end{rem}

\begin{rem}
The first examples of higher expanders came with the breakthrough work of Kaufman, Kazhdan and Lubotzky \cite{kaufman2014ramanujan}, who constructed simplicial $2$-complexes of bounded degrees which were topological $2$-expanders. This was later extended by Evra and Kaufman \cite{EK24} to a construction of bounded geometry $d$-expanders for any $d\geq 2$,  using quite refined combinatorial criteria for higher expansion. In both cases the spaces were derived from quotients of Bruhat-Tits buildings of simple $p$-adic groups. This settled Gromov's question in the setting of bounded geometry simplicial complexes.  Fox-Naor-Gromov-Lafforgue-Pach also studied higher expansion in the simplicial complex setting, in terms of overlap properties of maps of bounded degree hypergraphs to lower dimensional Euclidean spaces \cite{fox2012overlap}.    We also mention work by Abdurrahman-Adve-Giri-Lowe-Zung \cite{abdurrahman2024hyperbolic}  that established some higher expansion properties for families of hyperbolic 3-manifolds.

In future work corresponding to the later sections of the paper \cite{fl24} which this paper partially supersedes, we will apply the techniques from this paper to prove homological vanishing theorems for infinite volume higher rank locally symmetric spaces (see also \cite{cmw23,cmw25}), fixed point theorems for actions of higher rank lattices (see also \cite{bader2023higher},\cite{bader2025higher}, \cite{bader2026fixed}), and to establish non-abelian higher expansion properties for locally symmetric spaces.  
\end{rem}

\subsection{Minimal Submanifolds}
Let $M$ be a Riemannian manifold. A submanifold $S\subset M$ of dimension $m$ is \emph{minimal} if it is a critical point of the $m$-dimensional volume functional. In other words, for every compactly supported differentiable deformation $S_t, t\in [0,\varepsilon)$  with $S_0=S$ we have 
$$\frac{\partial}{\partial t}\mathcal H^m(S_t)|_{t=0}=0,$$ where $\mathcal H^m$ stands for the $m$-dimensional Hausdorff measure \cite[\S 1.2]{simongmt}.  This definition can be extended to stationary varifolds \cite[Chap. 4. \S 2.4]{simongmt}, which are a well behaved and useful generalization of minimal submanifolds. 

Minimal surfaces in hyperbolic $3$-manifolds have been studied and successfully used as a tool to study the ambient geometry \cite{calegari2006shrinkwrapping,gabai1997geometric,  Lackenby2002HeegaardST}. In contrast, the topic of low codimension minimal submanifolds of other locally symmetric spaces remains largely unexplored.

The proofs of Theorems \ref{mt-higherrank} and \ref{mt-waist} are based on  new monotonicity formulas for minimal submanifolds of low codimension in nonpositively curved symmetric spaces.  We use them to show that for $d$ small relative to the ambient dimension, codimension $d$ minimal submanifolds have to be large, of volume comparable to the volume of the ambient space.  More precisely, we prove the following:

\begin{thm}\label{mt-stationary}

\begin{enumerate} 
\item 

There are constants $c_1,c_2>0$ so that the following holds. Let $M$ be a compact Riemannian manifold with universal cover the symmetric space for a split non-compact simple real Lie group $G$.    Then any non-empty codimension $d \leq c_1\rank (G)$ stationary integral rectifiable varifold $S\subset M$ satisfies $\mathcal{H}^{\dim(M)-d}(S)\geq  c_2 \vol(M)$.  If $G=SL(n,\mathbb{R})$ then the same statement holds for $d \leq \lfloor n/8 \rfloor $.  

\item Let $M$ be a compact  octonionic hyperbolic $16$-manifold. There is a constant $c>0$ such that any non-empty stationary integral rectifiable $(16-d)$-varifold $S\subset M$ satisfies $\mathcal{H}^{16-d}(S)\geq c \vol(M)$ for $d=1,2$. 
\end{enumerate} 

\end{thm}

By Proposition \ref{prop-waistsweepout} below, the $d$-waist of a Riemannian manifold $M$ has as a lower bound the $\dim(M)-d$-volume of some stationary integer rectifiable varifold in $M$.  Theorem \ref{mt-waist} thus follows from Theorem \ref{mt-stationary}.



		
		\subsection{Power-Law Systolic Freedom}

Loewner\footnote{The result was first mentioned in print in the paper by Pu \cite{pu1952some} proving a systolic inequality for $\mathbb{RP}^2$, where it is credited to Loewner.} proved that any unit area Riemannian torus has a non-contractible loop of length at most $\sqrt{\frac{2}{\sqrt{3}}}$.  More generally, Gromov proved that for any unit volume compact aspherical Riemannian manifold $M$, there is an upper bound on the length of the smallest non-contractible loop in $M$ depending only on the dimension of $M$ \cite{gromov1983filling}.   This result is an example of a \textit{systolic inequality}, or inequality comparing the volume and the \textit{systole}, or infimal k-volume of a homologically or homotopically nontrivial k-cycle. 
The $(p,q)$-systolic ratio $\text{(p,q)-SR}(M)$ of a compact $(p+q)$-dimensional Riemannian manifold $M$ is defined to be 
\[
\text{(p,q)-SR}(M):= \inf_{P,Q  } \frac{\text{Vol}_p (P) \text{Vol}_q(Q)}{\text{Vol}(M)}, 
\]
\noindent where the infimum is taken over all rectifiable cycles $P$ and $Q$ that are nontrivial in respectively $H_p(M;R)$ and $H_q(M;R)$ for a coefficients ring $R$ that has been fixed in advance. Our convention is that $\text{(p,q)-SR}(M)$ is infinite if either of $H_p(M;R)$ or $H_q(M;R)$ vanishes.  We say that a sequence of bounded geometry Riemannian manifolds $M_n$ exhibits \textit{systolic freedom} if $\text{(p,q)-SR}(M_n) \rightarrow  \infty $ for some $p$ and $q$ and that $\{M_n\}$ exhibits \textit{power-law systolic freedom} if $\text{(p,q)-SR}(M_n) \gtrapprox \text{Vol}(M_n)^\beta$ for some $p$ and $q$, and $\beta>0$.  The study of systolic freedom was pioneered by Gromov (see \cite{gromovintersystolic} for a survey.)  

It follows from the results of the previous paragraph that there cannot be nontrivial examples of systolic freedom in two dimensional Riemannian manifolds. On the other hand, Gromov gave the first nontrivial examples of systolic freedom, which were homeomorphic to $S^3 \times S^1$ \cite{gromovintersystolic}. Freedman later gave examples in three dimensions with mod-2 coefficients \cite{freedman1999z2}.  Both of these came as a surprise at the time.   Freedman-Hastings gave the first nontrivial examples of Riemannian manifolds exhibiting power-law systolic freedom for $R=\mathbb{F}_2$ \cite{freedmanhastings}. Their construction has a combinatorial flavor; the Riemannian manifolds they construct are built from $\mathbb{F}_2$-chain complexes associated to quantum codes \footnote{Freedman-Hastings make the following comment regarding their examples: ``In retrospect it is not surprising that the most efficient constructions will necessarily have a combinatorial complexity beyond a geometer’s intuition, perhaps even requiring computer search, and are
better discovered within coding theory and then translated into geometry."}.  We are able to give a very different class of examples exhibiting power-law systolic freedom.

\begin{thm}\label{mt-sysfreedom}
Any sequence of principal congruence covers $M_n$ of a fixed compact octonionic hyperbolic manifold with volume tending to infinity exhibits power-law systolic freedom, for the coefficient ring $R$ equal to $\mathbb{Z}$ or $\frac{\mathbb{Z}}{k\mathbb{Z}}$.
\end{thm}
\begin{proof}
 \noindent  We claim that for $M_n$ as in the statement of the theorem
\[
\text{(2,14)-SR}(M_n) \gtrsim\text{Vol}(M_n)^\beta
\]

\noindent for some $\beta>0$.  We know by Theorem \ref{mt-stationary} part (2) together with the fact that any nontrivial homology class contains a stationary integral rectifiable varifold realizing the infimal volume of a cycle in that class (\cite{Fleming1966}, \cite{Federer1969}) that the 14-volume of any homologically nontrivial 14-cycle in $M_n$ is $ \gtrsim \text{Vol}(M_n)$.  On the other hand, it follows from \cite[Theorem A]{BelWein23} that the 2-volume of any nontrivial 2-cycle is $ \gtrsim \text{Vol}(M_n)^\beta$ for some $\beta>0$.  

\end{proof}

We comment that work by Guth and Lubotzky also used techniques from minimal surface theory, for example the Anderson monotonicity formula for minimal surfaces in hyperbolic manifolds \cite{anderson1982complete}, along with techniques from systolic geometry, to construct new quantum error correcting codes from arithmetic hyperbolic 4-manifolds \cite{guth2014quantum}.  

\subsection{Acknowledgements}  We thank Shmuel Weinberger for helpful discussions. We also thank Tali Kaufman, Yangyang Li and Hee Oh. Finally we thank an anonymous referee for helpful comments on the earlier paper that this paper partially supersedes.  MF was supported by the Dioscuri programme initiated by the Max Planck Society, jointly managed with the National Science
Centre in Poland, and mutually funded by the Polish Ministry of Education
and Science and the German Federal Ministry of Education and Research. BL was supported by the NSF under grant DMS-2202830.

\subsection{AI Disclosure} The authors used LLM tools to assist with literature searches, proof-checking, and learning purposes. An earlier version of the paper proved waist inequalities in the range $\rank(G)^{2/3}$.  LLM assistance allowed us to sharpen the estimates in Section \ref{sec-MonoHigherRank} and improve the range to linear in the rank.  
		
		\section{Background} \label{sec:background}
		\subsection{Varifolds and the first variation formula}\label{sec-varifolds}
		In this section we collect basic results and definitions concerning rectifiable varifolds. We fix an ambient Riemannian manifold $M$ of dimension $n$. Write $g$ for the metric tensor and $\nabla$ for the Levi-Civita connection. Let $\mathcal H^k$ denote the $k$-dimensional Hausdorff measure on $M$ \cite[Chap.1 \S 2]{simongmt}. Recalling that Lipschitz maps are differentiable a.e., for any $k$-dimensional Lipschitz embedded submanifold $S\subset M$, the measure $\mathcal H^k|_S$ coincides with the induced Riemannian volume on $S$.
		\begin{defn}[{\cite[Chap.3 \S 1]{simongmt}}]
			An $\mathcal H^k$-measurable subset of $M$ is countably $k$-rectifiable if it is a countable union of Lipschitz images of subsets of $\mathbb R^k$ and a set of $\mathcal H^k$-measure $0$.
		\end{defn}
		This definition actually implies a stronger characterization \cite[Chap 3. Lemma 1.2]{simongmt}. A countably $k$-rectifiable subset $A\subset M$ admits a decomposition
		\begin{equation}\label{eq-recfiabledecomp1}
			A=N_0\cup \bigcup_{i=1}^\infty N_i,
		\end{equation}
		where for $i\geq 1$ each $N_i$ is an $\mathcal{H}^k$-measurable subset of a $C^1$-embedded (possibly with boundary) $k$-submanifold and $\mathcal H^k(N_0)=0.$  Inductively defining $A_0=N_0$ and $A_i= N_i -\cup_{j=0}^{i-1} A_j$, we can write $A$ as a disjoint union of $A_0$ with the rectifiable sets $A_i$
        \begin{equation}\label{eq-recfiabledecomp}
			A=A_0\cup \bigsqcup_{i=1}^\infty A_i,
		\end{equation}
		\begin{defn}
			Let $A\subset M$ be an $\mathcal H^k$-measurable, countably $k$-rectifiable set such that $\mathcal H^k(A\cap K)$ is finite for all compact sets $K\subset M$. Then, for $x\in A$, a $k$-dimensional subspace $W\subset T_xM$ is said to be tangent to $A$ if for any compactly supported continuous function $f\colon T_x M \to \mathbb R_{\geq 0}$ we have 
			$$\lim_{\lambda\to 0}\lambda^{-k}\int_A f(\lambda^{-1}\log(y))d\mathcal H^k(y)=\int_W f(y)d\mathcal H^k(y).$$
		\end{defn}
		We write $T_x A$ for the tangent space of $A$ at $x$. In intuitive terms, zooming in at $x$ makes the $k$-dimensional Hausdorff measure on $A$ look more and more like the $k$-dimensional Lebesgue measure on $\exp(W)$. By \cite[Chap. 3 Thm 1.6]{simongmt}, if $A$ is countably $k$-rectifiable and has locally finite $\mathcal H^k$-measure, then $\mathcal H^k$-almost every point of $A$ admits a unique tangent plane. These points are called the \emph{regular points} of $A$. For the parts of $A$ which are $C^1$-submanifolds (see \ref{eq-recfiabledecomp1}) this definition of tangent space agrees with the ordinary definition of tangent space. In particular, the tangent space varies continuously after excluding a set of arbitrarily small $\mathcal H^k$-measure.
		
		\begin{defn}[{\cite[Chap. 4]{simongmt}}]
			Let $A\subset M$ be an $\mathcal H^k$-measurable, countably $k$-rectifiable set. Let $\theta$ be a non-negative $\mathcal H^k$-measurable function on $A$, such that $\int_{A \cap K}\theta d\mathcal H^k<\infty$ for any compact set $K.$ The rectifiable $k$-varifold $\underline{v}(A,\theta)$ attached to the pair $(A,\theta)$ is the equivalence class of pairs $(A',\theta')$ such that $A\setminus A', A'\setminus A$ are $\mathcal H^k$-measure $0$ and $\theta=\theta'$ $
			\mathcal H^k$-almost everywhere. A varifold is integral if $\theta$ is integer valued $\mathcal H^k$-almost everywhere. 
		\end{defn}
		 
		In this paper we work exclusively with integral rectifiable $k$-varifolds. We suppress the function $\theta$ from the notation and skip the adjective ``integral" from now on. Typically a rectifiable varifold will be denoted by the letter $S$. When integrating over a varifold $S$ represented by a pair $(A,\theta)$ we put 
		$$\int_S f(x)d \mathcal H^k(x):= \sum_{i=1}^{\infty } \int_{A_i} \theta(x)f(x) d\mathcal H^k(x) ,$$ 
		for any $f$ for which the right hand side converges, where $A= A_0 \cup (\sqcup_{i=1}^\infty A_i)$ as in (\ref{eq-recfiabledecomp}.) 
		Given a $C^1$-vector field $X$ on $M$ we define the associated flow $\Phi_t\colon M\to M, t\in \mathbb R$ by the condition $\frac{\partial}{\partial t}\Phi_t(x)=X(\Phi_t(x))$ and $\Phi_0(x)=x$ for all $x\in M$, see \cite[\S2.5]{perko2013differential} or \cite[Thm 17.8]{lee2012smooth} for the existence and uniqueness theorems.

		Since an integral rectifiable $k$-varifold $S$ is $\mathcal H^k$-almost everywhere a countable union of $\mathcal{H}^k$-measurable subsets of $C^1$-submanifolds \cite[Chap 3. Lemma 1.2]{simongmt}, we use the standard differential geometry conventions with the understanding that they are well defined $\mathcal H^k$-almost everywhere on $S$. 
		
		\begin{defn} \cite[\S4 Def 2.4]{simongmt} A rectifiable $k$-varifold $S\subset M$ 
			is \emph{stationary} if for every compactly supported $C^1$-vector field $X$ on $M$ with the associated flow $\Phi_t$, we have $\frac{\partial}{\partial t}\mathcal H^k(\Phi_t(B(R) \cap S))|_{t=0}=0$, for $B(R)$ a large ball containing the support of $X$.  
		\end{defn}
		
		We now define the divergence of a vector field relative to $S$. For any $C^1$ vector field $X$ on $M$ and regular point of $S$, put
		\begin{equation}\label{eq-divergence}
			{\rm div}_S X=\sum_{i=1}^k \langle \nabla_{e_i} X,e_i\rangle. 
		\end{equation}
		for $e_1, \ldots, e_k$ an orthonormal basis for the tangent space to $S$ at that point.  Note that when $X=\grad f, f\in C^2(M)$, then 
		\begin{equation}\label{eq-Hess}
			{\rm div}_S \grad f=\sum_{i=1}^k \langle \nabla_{e_i} \grad f,e_i\rangle = \sum_{i=1}^k \Hess f(e_i,e_i)=\tr \Hess f|_{TS}. 
		\end{equation}
		By \cite[Chap 2 \S5, Chap 4 \S2]{simongmt},
		\begin{equation}\label{eq-FV}
			\frac{\partial}{\partial t}\mathcal H^k(\Phi_t(S))|_{t=0}=\int_S {\rm div}_S X d\mathcal H^k 
		\end{equation}
		When $S$ is stationary, the first variation of the Hausdorff measure is zero, so
		\begin{equation}\label{eq-statFV}
			\int_S {\rm div}_S X d\mathcal H^k= 0, 
		\end{equation}
		for any compactly supported $C^1$-vector field $X$.
		In the particular case $X=\grad f, f\in C^2(M)$ we have ${\rm div_S} \grad f(x)=\tr \Hess f|_{T_xS}$ for $\mathcal H^k$-almost every $x\in S$. We can therefore write
		\begin{equation}\label{eq-statHessianBd}
			\frac{\partial}{\partial t} \mathcal H^k(\Phi_t(S)) |_{t=0}=\int_S {\rm div}_S \grad f d\mathcal H^k =\int_S \tr \Hess f|_{TS}\, d\mathcal H^k=0.
		\end{equation}
		We finish this section by giving an abstract form of the monotonicity estimate. 
		Let us fix a smooth non-increasing function $\chi\colon \mathbb R\to \mathbb R_{\geq 0}$ such that $\chi(t)=1$ for $t<0$, $\chi(t)>0$ for $t \in (0,1)$, $\chi(t)=0$ for $t\geq 1$ and $\chi'(t)\geq -2$ for all $t\in\mathbb R$. We put $\chi_r(t):=\chi(t-r).$ The function $\chi_r$ will serve as smooth cut-off throughout the paper. The exact choice of $\chi$ is not important, as long as it satisfies the properties listed above. 
		\begin{lem}\label{lem-abstmonotonicity}
			Let $f\colon M\to\mathbb R_{\geq 0}$ be a proper $C^2$ function with $\|\grad f\|\leq 1,$ with the property that $\tr \Hess f(x)|_W\geq \kappa$ for every $x\in M$ and every $k$-dimensional subspace $W\subset T_x M$. For any stationary $k$-varifold $S\subset M$ and $r>s>0$, we have 
			$$ \int_S \chi_r(f)d\mathcal H^k \geq e^{\kappa(r-s)}\int_S \chi_s(f)d\mathcal H^k.$$
		\end{lem}
		\begin{proof}
			Note that at any regular point of $S$ we have 
			\begin{align*} {\rm div}_S (\chi_r(f)\grad f)=& \chi_r(f){\rm div}_S\grad f +\chi'_r(f)\|\grad_S f\|^2\\
				=& \chi_r(f) \tr \Hess f|_{TS} + \chi_r'(f)\|\grad_S f\|^2.
			\end{align*}
			We apply the first variation formula to the vector field $\chi_r(f)\grad f.$
			\begin{align*}
				\int_S \chi_r(f(x))\tr \Hess f|_{T_xS}d\mathcal H^k(x)=& - \int_S \chi_r'(f(x))\|\grad_S f\|^2 d\mathcal H^k(x) \implies \\
				\kappa \int_S \chi_r(f(x))d\mathcal H^k(x)\leq& - \int_S \chi_r'(f(x))d\mathcal H^k(x)\\
				=& \frac{\partial}{\partial r}\int_S \chi_r(f(x))d\mathcal H^k(x).
			\end{align*}
			Note that for the second inequality we used $\chi_r'\leq 0$ and $\|\grad_S f\|\leq \|\grad f\|\leq 1.$ The conclusion of the lemma then follows by applying Gr\"onwall's inequality \cite{gronwall1919note} to the function $r\mapsto \int_S \chi_r(f(x))d\mathcal H^k(x).$
		\end{proof}
		We will need one more fact about stationary varifolds in symmetric spaces.
		\begin{lem}\label{lem-smallballmono}
			Let $X$ be a symmetric space. Then for any $r$ there is $\delta(r)>0$ depending on $X$ so that for any stationary varifold $S\subset X$ and any regular point $x\in S$ we have
			$$\mathcal H^k(S\cap B(x,r))\geq \delta(r).$$

		\end{lem}
		\begin{proof}

			By \cite[Section 6]{montezumanotes}, there are $r_0, C$ depending only on $X$ so that any stationary integral $k$-varifold $S\subset X$, $x \in S$ regular and $r<r_0$ we have
			\[
			\frac{d}{dr}\left( \log \frac{e^{Cr^2}\mathcal H^k(B(x,r) \cap S)}{r^k} \right) \geq 0.  
			\]
			Since the limit $\lim_{r \to 0} \mathcal H^k(B(x,r)\cap S)/r^k$ is a positive integer multiple of the volume of the k-dimensional unit ball, the conclusion follows.  
			
		\end{proof}

        \subsection{Generalized Piecewise-Linear Maps} \label{sec-PL}


Let $M$ be a compact smooth manifold that has been given a triangulation $\mathcal{T}$. 
A triangulation of $M$ is a finite simplicial complex $\mathcal{T}$ together with a homeomorphism $\mathcal{T} \rightarrow M$ that restricts to a smooth embedding on the interior of each simplex in $\mathcal{T}$. We will treat simplices in $\mathcal{T}$ as interchangeable with their embedded images in $M$.  Every smooth manifold has a triangulation \cite[Theorem 10.6]{munkres2016elementary}. As in previous sections we write $M^{(m)}, M^{[m]}$ for the $m$-skeleta and sets of $m$-simplices respectively.   



We work with the following class of maps, which one can check include piecewise linear maps.  We refer the reader to \cite{bryantpltopology}[Section 2] for the definition of piecewise linear maps.  

\begin{defn} \label{PLdefn}
	We say that a map $F$ from a smooth manifold $M$  to $\mathbb{R}^d$ is \textit{generalized piecewise linear} if there is a triangulation $\mathcal{T}$ of $M$ and a map $\Phi:M \rightarrow \mathbb{R}^n$ so that the following holds.  First we require that $\Phi$ restricted to each closed simplex in $\mathcal{T}$ is a $C^1$-diffeomorphism onto a simplex in $\mathbb{R}^n$. Second we require, for every simplex $\sigma$ in $(M,\mathcal{T})$, that the map $F \circ (\Phi|_{\sigma})^{-1}: \Phi(\sigma) \rightarrow \mathbb{R}^d $  is the restriction of an affine linear map $\mathbb{R}^n \rightarrow \mathbb{R}^d$.  
\end{defn}

We note that it is possible to construct many generalized piecewise linear maps $M \rightarrow \mathbb{R}^d$ by fixing a triangulation $\mathcal{T}$ of $M$, choosing some map $\Phi$ as above, choosing where the vertices of $\mathcal{T}$ (i.e. $M^{(0)}$) map to, and then extending linearly over each simplex. We also note that it is possible to construct many such maps $\Phi$ by choosing where in $\mathbb{R}^n$ to send the vertices of $\mathcal{T}$ and then extending inductively over skeleta.



\begin{defn}
	Let $M$ be a smooth manifold.  Then we say that a generalized piecewise linear map $M \rightarrow \mathbb{R}^d$ with respect to some triangulation $\mathcal{T}$ is \textit{generic} if for no $d$-dimensional simplex in $\mathcal{T}$ is it the case that its $d+1$ vertices  map to a set of points that are contained in an affine hyperplane in $\mathbb{R}^d$.  
\end{defn}

We give now the precise definition of d-waist that we will use in this paper. 

\begin{defn} Let $X$ be a compact Riemannian manifold of dimension $n>d$. Let $f\colon X\to \mathbb{R}^d$ be a continuous map. The \emph{waist} of $f$ is 
 $w(f):=\sup_{y\in \mathbb{R}^d} \mathcal H^{n-d}(f^{-1}(y))$, where $\mathcal H^{n-d}$ stands for the $(n-d)$-dimensional Hausdorff measure. The \emph{$d$-waist} of $X$ is then defined as $w_d(X):=\inf_{f\colon X\to \mathbb R^d} w(f)$, where the infimum is taken over the set of generic generalized piecewise linear maps.  
\end{defn}

The following proposition is an application of Almgren-Pitts min-max theory, and is proved in Appendix \ref{sec-sweepouts}. 
\begin{prop} \label{prop-waistsweepout}
For $X$ as in the previous definition, there is a non-empty stationary integral rectifiable varifold in $X$ whose $n-d$-dimensional Hausdorff measure is a lower bound for the $d$-waist of $X$.  
\end{prop}

		\subsection{Symmetric spaces and spherical functions}\label{sec-symspaces}
		We follow the standard notations of \cite{knapp1996lie,gv88}, with the exception of Iwasawa decomposition. We will be using $NAK$ instead of $KAN$, which is the order of factors in \cite{gv88} and most sources \footnote{The reason is that we want the central Iwasawa coordinate $H\colon G\to \mathfrak a$ to be a right $K$-invariant function.}. Throughout the paper, $G$ will be reserved for a noncompact simple real Lie group. We let $K$ be a maximal compact subgroup of $G$ fixed by a Cartan involution $\Theta$. The symmetric space of $G$ is the quotient $X=G/K$. Write $\mathfrak g, \mathfrak k$ for the Lie algebras of $G,K$ respectively and let $\mathfrak s=\{Y\in \mathfrak g\mid \Theta Y=-Y\}$. Let $A$ be a maximal split torus of $G$ stabilized by $\Theta$ and let $\mathfrak a$ be the Lie algebra of $A$. We necessarily have $\mathfrak a\subseteq \mathfrak s$. The tangent space $T_KX$ (at the identity coset of $K$) is $\mathfrak g/\mathfrak k$. The latter is identified with $\mathfrak s$ via the map $\mathfrak g/\mathfrak k\ni Y+\mathfrak k\mapsto \frac{1}{2}(Y-\Theta Y)\in \mathfrak s$. The Killing form on $\mathfrak g$ is given by 
		$$B(Y,Z)=\tr \ad Y\ad Z.$$
		It descends to a positive definite symmetric bilinear form 
		$$\langle Y,Z\rangle_{\rm Kill}:= B(Y,Z) \hspace{3mm} Y,Z\in \mathfrak s.$$
		This extends to a unique left $G$-invariant Riemannian metric on  $X$, denoted $g_{\rm Kill}$.  We will write $\langle \cdot,\cdot\rangle, g$ for unspecified (but fixed) real multiples of $\langle \cdot,\cdot\rangle_{\rm Kill}, g_{\rm Kill}.$ The proofs and statements in the paper are written in a way that does not depend on the chosen normalization. For example, in the case of rank one symmetric spaces, we can choose $g$ so that the maximum sectional curvature is $-1$ as opposed to keeping the metric induced by the Killing form.
		
		We will make no distinction between functions on $X$ and right-$K$ invariant functions on $G$. In this way, given $f\colon X\to \mathbb R$ we write $f(g)$ for $f(gK), gK\in X$. 
		
		Let $P$ be a minimal parabolic subgroup of $G$ containing $A$. The group $P$ decomposes as $P=MAN$ with $M=P\cap K=Z(A)\cap K$ and $N$ being the unipotent radical of $P.$ This decomposition is often called the Langlands decomposition. Let $\mathfrak p, \mathfrak m, \mathfrak n$ be the Lie algebras of $P,M,N$ respectively.
		The non-negative and positive Weyl chambers are defined in the standard way.
		\begin{align*}
			A^+=& \overline{\{a\in A\mid a^{-t} n a^t \to 1 \text{ as } t\to\infty \text{ for all }n\in N\}},\\
			\mathfrak a^+=&\log A^+\subset \mathfrak a \quad\text{ (the non-negative Weyl chamber)},\\
			\mathfrak a^{++}=&\,{\rm int}\, \mathfrak a^+ \quad\text{ (the positive Weyl chamber)}.
		\end{align*}
		Let $\mathfrak a^*:=\Hom(\mathfrak a,\mathbb R)$ and $\mathfrak a^*_\mathbb C:=\Hom(\mathfrak a,\mathbb C).$ Let $\lambda\in \mathfrak a^*$. We put 
		$$\mathfrak g_\lambda:=\{Y\in \mathfrak g\mid [H,Y]=\lambda(H)Y \text{ for } H\in \mathfrak a\}.$$
		If $\mathfrak g_\lambda\neq 0$ and $\lambda \neq 0$, then we call $\lambda$ a \emph{root} and put $m_\lambda:=\dim \mathfrak g_\lambda.$ The number $m_\lambda$ is called the \emph{multiplicity} of $\lambda$ and $\mathfrak g_\lambda$ is the \emph{root subspace} of $\lambda.$ We typically reserve letters $\alpha,\beta$ for the roots and $\lambda,\xi$ for general elements of $\mathfrak{a}^*$.
		Let $\Phi$ be the set of roots and let $\Phi^+:=\{\alpha\in \Phi\mid \mathfrak g_\alpha\subseteq \mathfrak n\}$ be the set of positive roots. We write $\tilde \Phi,\tilde \Phi^+$ for the multisets of roots and positive roots, where $\alpha$ has multiplicity $m_\alpha.$
		Let $\rho$ be the half-sum of positive roots $$2\rho:=\sum_{\alpha\in \Phi^+}m_\alpha \alpha= \sum_{\alpha\in\tilde\Phi^+}\alpha \in \mathfrak a^*.$$
		The \emph{simple roots} are those positive roots which cannot be expressed as positive integer combinations of other positive roots. The set of simple roots will be denoted $\Delta$, their number is equal to the rank of $G$ and they form a basis of $\mathfrak a^*$. Moreover, any positive root can be expressed as a positive integer combination of simple roots. We say that a root $\alpha$ is \emph{reduced} if $\alpha/2\not\in\Phi.$
		The \emph{Weyl group} is the quotient $N_K(A)/Z_K(A)$ of the normalizer in $K$ of $A$ by the centralizer $Z_K(A)$ of $A$ in $K$.  Note that $Z_K(A)$ is equal to the group $M$ above appearing in the Langland's decomposition, and that it is finite if $G$ is split. It acts on $A,\mathfrak a$ and $\mathfrak a^*$. The non-negative Weyl chamber $\mathfrak a^+$ is a fundamental domain for the action on $\mathfrak a$. The action is free on the union $\bigcup_{w\in W}w\mathfrak a^{++}.$
		The dual Weyl chamber is defined as $(\mathfrak a^*)^+:=\{\lambda\in \mathfrak a^*\mid \lambda \text{ is positive on}  \hspace{1mm} \mathfrak{a}^{++} \}$.  The
		 \emph{Iwasawa decomposition}\footnote{In other sources one might find a different order of factors. We opt to have $K$ on the right so that the Iwasawa coordinate $H$ descends to a function on $X=G/K$.} is the identity $G=NAK$. This decomposition is unique and the map sending $(n,a,k)\to nak$ is a diffeomorphism \cite[p.63]{gv88}. We define the function $H\colon G\to \mathfrak a$ by 
		 $$g=ne^{H(g)}k, k\in K,n\in N.$$ We can integrate in Iwasawa coordinates using \cite[Prop. 2.4.2]{gv88}
		 $$\int_G f(g) dg=c_{\rm{Iwa}}\int_N\int_{\mathfrak a}\int_K f(ne^Hk) e^{-2\rho(H)} dk dH dN,$$ where $c_{\rm{Iwa}}$ depends only on the choice of Haar measures on $K,\mathfrak a, N.$
		The \emph{Cartan decomposition} is the identity $G=KA^+K$. We define the function $a\colon  G\to \mathfrak a^+$ by putting $$g=k_1e^{a(g)}k_2, k_1,k_2\in K.$$ The map $\phi_{KAK}:(k_1,a,k_2)\to k_1ak_2$ descends to an analytic diffeomorphism $ (K \times \mathfrak a^{++} \times K) / M \rightarrow K exp (\mathfrak{a}^{++}) K$ and the image is an open dense subset of $G$.  Here
        \[
        m \cdot (k_1,H,k_2) = (k_1 m, H, m^{-1}k_2)
        \]  To integrate in Cartan coordinates we use the formula \cite[Prop. 2.4.6]{gv88}
		\begin{equation} \label{eq-sphericalcoord}\int_G f(g)dg =c_{\rm{Car}}\int_K\int_K\int_{\mathfrak a^+}f(k_1e^Hk_2) J(H)dHdk_1dk_2,
        \end{equation}
        where $J(H):=\prod_{\alpha\in\Phi^+}\sinh(\alpha(H))^{m_\alpha}$ and $c_{\rm{Car}}$ depends on the choice of Haar measure $dk_1,dk_2,dH$ on $K,K,\mathfrak a.$ Unless otherwise mentioned we will always choose the Haar probability measure on $K$. The fact that the map $\phi_{KAK}$ above is not actually a diffeomorphism, but descends to a diffeomorphism on an open dense set after quotienting by the group $M$, is absorbed in the constant $c_{\rm{Car}}$.  
		
		A function $\phi\colon G\to\mathbb C$ is a \emph{spherical function} if $\phi(k_1gk_2)=\phi(g)$ for all $g\in G, k_1,k_2\in K.$ Let $V$ be a Hilbert space and let $\pi\colon G\to \mathscr U(V)$ be a unitary representation of $G$. Let 
		$$V^K:=\{v\in V\mid \pi(k)v=v, k\in K\},$$ be the subspace of $K$ fixed vectors. For $v,w\in V^K$, the matrix coefficient
		$g\mapsto \langle \pi(g)v,w\rangle$ is a spherical function. If $(\pi, V)$ is irreducible and $V^K\neq 0$, we know that $\dim V^K=1$ \cite[Prop. 1.5.8]{gv88}.
		The elementary spherical functions are the matrix coefficients $g\mapsto \langle \pi(g)v,v\rangle$, where $(\pi,V)$ is irreducible, $v\in V^K$ and $\|v\|=1.$
		Thanks to Harish-Chandra, we have an explicit formula for the elementary spherical functions.
		\begin{thm}\label{thm-ElmSph}\cite[Prop 3.1.4 and Thm. 3.2.3]{gv88}
			Let $(\pi, V)$ be an irreducible unitary representation with $V^K\neq 0$. Let $v\in V^K$, $\|v\|=1$. There exists a $\lambda\in \mathfrak a^*_\mathbb C$ such that 
			$$\langle \pi(g)v,v\rangle = \int_K e^{(\lambda+\rho)(H(kg))}dk=:\varphi_\lambda(g).$$  
			The character $\lambda$ is determined uniquely up to the $W$-action and called the \emph{infinitesimal character} of $\pi.$ Characters $w\lambda, w\in W$ give rise to the same elementary spherical function.
		\end{thm}

		We will make use of the following identities and estimates for elementary spherical functions \cite[Thm 4.6.4]{gv88}. Let $d$ be the number of reduced positive roots counted without multiplicity. It is always bounded by $|\Phi^+|\leq |\tilde \Phi^+|=\dim X-\dim \mathfrak a.$
		\begin{align} \varphi_0(g)\leq& C e^{-\rho(a(g))}(1+\|a(g)\|)^{d}, \label{eq-basicMCbounds}\\
			\varphi_{w\rho}(g)=&1 \text{ for all }w\in W, g\in G,  \nonumber
		\end{align}
		\noindent where $C=C(G)$ is a uniform constant. For the first formula, keep in mind that the Iwasawa coordinates $H=H_{NAK}$ and $H_{KAN}$ relative to respectively the $NAK$ and $KAN$ decompositions satisfy $H_{NAK}(g)=-H_{KAN}(g^{-1})$. The second formula can be recovered from the integral representation in Theorem \ref{thm-ElmSph} for $\lambda=-\rho$ and $w=1$, and the fact that $\varphi_{w\lambda}=\varphi_{\lambda}$ for all $\lambda\in \mathfrak a^*_\mathbb C, w\in W$ \cite[Prop. 3.2.2]{gv88}. 
		\begin{lem}\label{lem-MCconvexity}
			The map $\lambda\mapsto \log \varphi_\lambda(g)$ is convex on $\mathfrak a^*.$
		\end{lem}
		\begin{proof}
			To prove our claim we just need to show that the Hessian of $\log \varphi_\lambda(g)$ is non-negative definite.  Let $\xi\in \mathfrak a^*\cong T_\lambda \mathfrak a^*$. By the Cauchy-Schwartz inequality,
			\begin{align*}\frac{\partial^2}{\partial \xi\partial \xi}\log \varphi_\lambda(g)=&\varphi_\lambda(g)^{-2}\left(\int_K e^{(\lambda+\rho)(H(kg))}dk\int_K \xi(H(kg))^2 e^{(\lambda+\rho)(H(kg))}dk\right.\\ & \left.-\int_K \xi(H(kg))e^{(\lambda +\rho)(H(kg))}dk\int_K \xi(H(kg))e^{(\lambda + \rho)(H(kg))}dk\right)\geq 0.
			\end{align*} 
		\end{proof}
		\subsection{Decay of matrix coefficients in rank one groups}\label{sec-decay}
		Throughout this section we let $G$ be a rank one simple real Lie group. Let $\alpha$ be the simple root. Then the set of positive roots is either $\{\alpha\}$ or $\{\alpha, 2\alpha\}.$ The multiplicities are listed in the table below, covering all isogeny classes of real rank one groups
		\begin{table}[ht]
			\centering
			\begin{tabular}{ c c c c c}
				$G$ & $X$ & $\dim X$ & $m_\alpha$ & $m_{2\alpha}$ \\ \hline
				$\SO(n,1)$ & $\mathbb H^n_\mathbb R$ & $n$ & $n-1$ & $0$\\
				$\SU(n,1)$ & $\mathbb H^n_\mathbb C$ & $2n$ & $2n-2$ & $1$\\
				$\Sp(n,1)$ & $\mathbb H^n_\mathbb H$ & $4n$ & $4n-4$ & $3$\\
				$F_4^{(-20)}$ & $\mathbb H^2_\mathbb O$ & $16$ & $8$ & $7$\\
			\end{tabular}
			\caption{Multiplicities of roots in rank one groups.}
			\label{tab-multiplicities}
		\end{table}
		\begin{thm}[{\cite{kostant1969existence}}] Let $G$ be a rank one simple real Lie group and let $(\pi,V)$ be a nontrivial irreducible unitary representation with $V^K\neq 0.$ Then, the infinitesimal character $\lambda$ satisfies 
			$$\left| \Re \frac{\langle \lambda,\alpha\rangle}{\langle \alpha, \alpha\rangle} \right|\leq \begin{cases} \frac{m_\alpha}{2}+1 & \text{if} \hspace{1mm} m_{2\alpha}\neq 0\\ \frac{m_\alpha}{2} & \text{ if }m_{2\alpha}=0 \end{cases}.$$
			Furthermore, if $\lambda\not\in i\mathfrak a^*$ then $\lambda\in \mathfrak a^*$ (it's either real or purely imaginary). 
		\end{thm}
		\begin{lem}\label{lem-rankonedecay} Let $G$ be either $\Sp(n,1)$ for $n>1$ or $F_4^{(-20)}$. Let $\varphi_\lambda$ be an elementary spherical function corresponding to a nontrivial unitary representation. Then
			$$| \varphi_\lambda(k_1e^Hk_2)|\leq C e^{(1-m_{2\alpha})\alpha(H)}(1+\|H\|), \text{ for } H\in \mathfrak a^+, k_1,k_2\in K.$$
		\end{lem}
		\begin{proof}
			The function is spherical, so we might as well assume $k_1=k_2=1$. If $\lambda\in i\mathfrak a^*$ then by (\ref{eq-basicMCbounds})
			$$|\varphi_\lambda(e^H)|\leq |\varphi_0(e^H)|\leq C e^{-\rho(H)}(1+\|H\|)\leq C e^{(1-m_{2\alpha})\alpha(H)}(1+\|H\|).$$
			
			Otherwise, by Kostant's theorem, $\lambda= t_1\alpha,$ with $|t_1|\leq \frac{m_\alpha}{2}+1.$ Applying the nontrivial Weyl group element to $\lambda$ (which in rank one case is the multiplication by $-1$) does not change the spherical function, so we can assume $t_1\geq 0$. We have $\lambda=\frac{(\rho-\lambda)(H)}{\rho(H)}0+\frac{\lambda(H)}{\rho(H)}\rho$, for any non-zero $H\in \mathfrak a$. Lemma \ref{lem-MCconvexity} yields
			\begin{align*}
				\log|\varphi_\lambda(e^H)|\leq& \frac{\lambda(H)}{\rho(H)}\log |\varphi_\rho(e^H)|+\frac{(\rho-\lambda)(H)}{\rho(H)}\log |\varphi_0(e^H)|\\
				\leq& \frac{(\rho-\lambda)(H)}{\rho(H)}\log(Ce^{-\rho(H)}(1+\|H\|))\\
				\leq& (\lambda-\rho)(H)+ \log (1+\|H\|) + \log C .
			\end{align*}
			We have $\rho=(\frac{m_\alpha}{2}+m_{2\alpha})\alpha,$ so $(\lambda-\rho)(H)\leq (1-m_{2\alpha})\alpha(H).$
		\end{proof} 
	
	\subsection{Decay of matrix coefficients for \iffalse $\SL(n,\mathbb{R})$ and\fi  simple higher rank groups }\label{sec-decayhigherrank}

	We state now the result on decay of matrix coefficients due to Oh that we will need.  This is analogous to the work by Kostant above. 
    A subset $\Theta$ of roots in $\Phi^+$ is strongly orthogonal if $\pm\alpha\pm\beta\not\in\Phi$ for all distinct $\alpha,\beta\in \Theta.$ Below we list these maximal strongly orthogonal subsets for the classical root systems of types $A_n,B_n$, $C_n$, and $D_n$. We express the roots in terms of the standard orthonormal basis $\{e_1, e_2, \dots, e_n\}$ of $\mathbb{R}^n$. Recall that $\SL(n,\mathbb R)$ is the split real Lie group with root system $A_{n-1}.$ Let $\theta$ be the half sum of the positive roots in a maximal strongly orthogonal subset $\Theta$.
    
    \textbf{Type $A_{n-1}$} In the case of $\SL(n,\mathbb{R})$, we have $\Phi=\{e_i-e_j\mid i\neq j=1,\ldots,n\}$. Write $\alpha_{i,j}:=e_i-e_j.$ A maximal strongly orthogonal subset is given by $\Theta=\{e_i-e_{n+1-i}\mid i=1,\ldots \lfloor n/2\rfloor\}.$ The half-sum is 
\[
2\theta= \sum_{i=1}^{\lfloor n/2 \rfloor} \alpha_{i,n+1-i}= e_1 +.. + e_{\lfloor n/2 \rfloor} - e_{\lceil n/2 \rceil +1 }-..-e_n. 
\]

\textbf{Type $B_n$}
The root system of type $B_n$ (corresponding to the orthogonal group $\SO(2n+1)$) is given by:
\[ \Phi = \{\pm e_i \pm e_j \mid 1 \leq i < j \leq n\} \cup \{\pm e_i \mid 1 \leq i \leq n\} \]
The maximal strongly orthogonal subset $\Theta$ we choose depends on the parity of $n$. If $n = 2k$ is even: 
    \[ \Theta = \{ e_1 - e_2, e_1 + e_2, e_3 - e_4, e_3 + e_4, \dots, e_{2k-1} - e_{2k}, e_{2k-1} + e_{2k} \} \]
If $n = 2k+1$ is odd:
    \[ \Theta = \{ e_1 - e_2, e_1 + e_2, e_3 - e_4, e_3 + e_4, \dots, e_{2k-1} - e_{2k}, e_{2k-1} + e_{2k}, e_{2k+1} \} \]
In both cases, the size of the maximal strongly orthogonal subset is $n$ and the half-sum of $\Theta$ is:
\[
\theta=
\begin{cases}
\displaystyle\sum_{j=1}^{k} e_{2j-1}, & n=2k,\\[6pt]
\displaystyle\sum_{j=1}^{k} e_{2j-1}+\frac12 e_{2k+1}, & n=2k+1.
\end{cases}
\]

\textbf{Type $C_n$}
The root system of type $C_n$ (corresponding to the symplectic group $\Sp(2n)$) is given by:
\[ \Phi = \{\pm e_i \pm e_j \mid 1 \leq i < j \leq n\} \cup \{\pm 2e_i \mid 1 \leq i \leq n\} \]
A maximal strongly orthogonal subset $\Theta$ is given by the long roots:
\[ \Theta = \{ 2e_1, 2e_2, 2e_3, \dots, 2e_n \} \]
The size of this maximal strongly orthogonal subset is $n$ and the half-sum of $\Theta$ is:
\[ \theta= \sum_{i=1}^{n} e_i\]

\textbf{Type $D_n$}
The root system of type $D_n$ (corresponding to the orthogonal group $\SO(2n)$) is given by:
\[ \Phi = \{\pm e_i \pm e_j \mid 1 \leq i < j \leq n\} \]
A maximal strongly orthogonal subset $\Theta$ is constructed by taking disjoint pairs:
\[ \Theta = \{ e_1 - e_2, e_1 + e_2, e_3 - e_4, e_3 + e_4, \dots, e_{2k-1} - e_{2k}, e_{2k-1} + e_{2k} \} \]
where $k = \lfloor n/2 \rfloor$. The half-sum of $\Theta$ is 
\[ \theta=\sum_{j=1}^{k} e_{2j-1}.\]

Oh proved the following theorem \cite[Theorem 1.1, 4.10, 4.11]{oh02} :  
	
	\begin{thm} \label{thm-decayslnr}
	Let $G$ be a non-compact simple real Lie group and $K$ be a maximal compact subgroup of $G$. Suppose that $\rank(G) \geq 2$. Then for any mean-zero $K$-invariant function $f \in L^2( \Gamma \backslash G)$, any $v \in \mathfrak{a}^+$, and any $\delta>0$, 
	
	\begin{equation} 
		|\langle (\text{exp } v) f, f \rangle| \leq d_{\delta} e^{-(1-\delta)\theta(v)} ||f||_2^2, 
		\end{equation} 
	where $d_{\delta}$ depends only on $\delta$.  
		\end{thm}

	\section{Hessians in symmetric spaces}\label{sec-Hess}

Let $F$ be a $C^2$, $K$-invariant function on $X=G/K$, and write
$F(ke^HK)=\varphi(H)$. At a regular point $x=ke^HK$, left translation
by $ke^H$ identifies the orthogonal decomposition
\[
 \mathfrak{s}=\mathfrak{a}\oplus
 \bigoplus_{\alpha\in\Phi^+}\mathfrak{s}_\alpha,
 \qquad
 \mathfrak{s}_\alpha=(\mathfrak{g}_\alpha\oplus
 \mathfrak{g}_{-\alpha})\cap\mathfrak{s},
\]
with an orthogonal decomposition of $T_xX$. Denote the image of
$\xi\in\mathfrak{a}$ by $U_\xi$, and choose an orthonormal basis
$V_{\alpha,1},\ldots,V_{\alpha,m_\alpha}$ in each root summand.
As before, $H_\alpha$ is the metric dual of $\alpha$:
$\langle H_\alpha,\xi\rangle=\alpha(\xi)$.

\begin{prop}\label{prop-CartanHess}
For $H\in\mathfrak{a}^{++}$, the Hessian of $F$ at $ke^HK$ is given by
\begin{align*}
 \operatorname{Hess}F(U_\xi,U_\eta)
   &=D^2\varphi(H)(\xi,\eta),\\
 \operatorname{Hess}F(U_\xi,V_{\alpha,i})&=0,\\
 \operatorname{Hess}F(V_{\alpha,i},V_{\beta,j})
   &=\delta_{\alpha\beta}\delta_{ij}\,
     \coth(\alpha(H))\,
     \langle\nabla\varphi(H),H_\alpha\rangle.
\end{align*}
\end{prop}

\begin{proof}
Put $M=Z_K(\mathfrak{a})$. In polar coordinates
$K/M\times\mathfrak{a}^{++}\longrightarrow X^{++}$,
$(kM,H)\mapsto ke^HK$, the metric in polar coordinates is
\[
g=dH^2+\sum_{\alpha\in\Phi^+}\sinh^2(\alpha(H))\,q_\alpha;
\] where $q_\alpha$ is an $H$-independent quadratic form on the angular
root summand, extended by zero on the other summands. For the proof we refer to the computation\footnote{We also derived these formulas directly from the definition of Levi-Civita connection in the earlier version of the manuscript \cite{fl24}} of the differential of the polar map in \cite[Chapter I, proof of Theorem 5.8]{helgason2022groups}.

Since $\nabla F=\nabla\varphi(H)$ is horizontal, the identity
$2\operatorname{Hess}F=\mathcal L_{\nabla F}g$ gives the horizontal
block $D^2\varphi$, zero mixed blocks, and angular blocks
\[
 \frac12\nabla\varphi(H)
       \bigl[\sinh^2(\alpha(H))\bigr]q_\alpha
 =\alpha(\nabla\varphi(H))\coth(\alpha(H))\,
       g\big|_{\mathfrak{s}_\alpha}.
\]
Evaluating in the stated orthonormal bases proves the formulas.
\end{proof}

We suppress the multiplicity index when $G$ is split; equivalently,
the vectors $V_\alpha$ may be indexed by the multiset
$\widetilde\Phi^+$ used above. At singular points, the Hessian is
obtained by continuity.

		\section{Monotonicity estimates in rank one symmetric spaces}\label{sec-mono}
	
	In Section \ref{sec-varifolds}, we have fixed a smooth non-increasing function $\chi\colon \mathbb R\to \mathbb R_{\geq 0}$ such that $\chi(t)=1$ for $t<0$, $\chi(t)=0$ for $t>1$ and $\chi'(t)\geq -2$ for all $t\in\mathbb R$. We have also defined $\chi_r(t):=\chi(t-r).$ We keep this choice and notation fixed in the remainder of the paper.
	
	Let $G$ be a rank one simple real Lie group. We keep the notations from Section \ref{sec-symspaces} and \ref{sec-decay}.  
	Let  $\kappa\colon \{1,\ldots, \dim X-1\}\to \mathbb R_{\geq 0}$, be given as
	\begin{table}[ht]
		\centering
\begin{tabular}{ll|c|c}
    & & $X$ & $\dim X$ \\ \hline
    $\kappa(k)=k-1$
    & $k=1,\ldots,n-1$
    & $\mathbb H_{\mathbb R}^n$ & $n$ \\

    $\kappa(k)=k-1$
    & $k=1,\ldots,2n-1$
    & $\mathbb H_{\mathbb C}^n$ & $2n$ \\

    $\kappa(k)=\begin{cases}
        k-1\\
        4n-4+2i
    \end{cases}$
    & $\begin{array}{l}
        k=1,\ldots,4n-3\\
        k=4n-3+i,\quad i=1,2
    \end{array}$
    & $\mathbb H_{\mathbb H}^n$ & $4n$ \\

    $\kappa(k)=\begin{cases}
        k-1\\
        8+2i
    \end{cases}$
    & $\begin{array}{l}
        k=1,\ldots,9\\
        k=9+i,\quad i=1,\ldots,6
    \end{array}$
    & $\mathbb H_{\mathbb O}^2$ & $16$ \\
\end{tabular}

		\caption{$\kappa$ in rank one groups.}
		\label{tab-kappa}
	\end{table}
	
	We define a spherical function $\ff\colon G\to \mathbb R_{\geq 0}$ by 
	\begin{equation}\label{defn-ffdef}
		\ff(g):= \frac{1}{2\|\alpha\|}\log (2\cosh(2\alpha(a(g)))), 
	\end{equation} 
	where $\alpha\in \Phi^+$ is the simple root. In what follows we assume that the metric on octonionic hyperbolic space was chosen so that $0<f(e)= \log 2/ (2 ||\alpha||)<1$, so that $\chi_0(f)$ is not everywhere vanishing.   This function is a smooth approximation to the distance to the origin $K\in X$, which is given by 
	$$d(gK,K)=\|a(g)\|=\frac{\alpha(a(g))}{\|\alpha\|}.$$
	
	\begin{lem}\label{lem-HessianRk1}
		For any $x\in X$, any $k$-dimensional subspace $W\subset T_x X,$ we have 
		$$ \tr \Hess \ff|_W\geq \kappa(k)\|\alpha\|.$$
	\end{lem}
	\begin{proof}
		Let $H_1\in \mathfrak a^+$ be the element with $\|H_1\|=1.$ By Proposition \ref{prop-CartanHess}, the eigenvalues of $\Hess \ff$ at a point $ke^HK\in X, k\in K, H=tH_1\in \mathfrak a^{++}$ are 
		\begin{align*}
			\frac{\partial^2}{(\partial H_1)^2}\frac{1}{2\|\alpha\|}\log(2\cosh(2\alpha(H)))=& \frac{1}{2||\alpha||}\frac{\partial^2}{(\partial t)^2}\log\cosh(2t||\alpha||)= \frac{\partial}{\partial t}\frac{\sinh(2t||\alpha||)}{\cosh(2t||\alpha||)}\\
			=&2\|\alpha\|(1-\tanh^2(2t||\alpha||))>0, \text{ with multiplicity }1,\\
			\frac{\cosh(\alpha(H))}{\sinh(\alpha(H))}\frac{\partial}{\partial H_\alpha}\frac{1}{2||\alpha||}\log\cosh(2\alpha(H))=&\|\alpha\| \frac{\cosh(\alpha(H))\sinh(2\alpha(H))}{\sinh(\alpha(H))\cosh(2\alpha(H))}\\
			\geq& \|\alpha\| \text{ with multiplicity }m_\alpha,\\
			\frac{\cosh(2\alpha(H))}{\sinh(2\alpha(H))}\frac{\partial}{\partial H_{2\alpha}}\frac{1}{2||\alpha||}\log\cosh(2\alpha(H))=& 2\|\alpha\| \text{ with multiplicity }m_{2\alpha}.\\
		\end{align*}
		The trace $\tr \Hess \ff(x)|_W$ is bounded from below by the sum of bottom $k$ eigenvalues of $\Hess \ff(x)$. Therefore, for $k=1$ the lower bound is $0$, for $k=1+i, i\leq m_\alpha$, the lower bound is $i\|\alpha\|$ and for $k=1+m_{\alpha}+i, i<m_{2\alpha}$, the lower bound is $(m_\alpha+2i)\|\alpha\|$. Analyzing these bounds case by case, we arrive at the lower bound $\tr \Hess \ff(x)|_W\geq \kappa(k)\|\alpha\|$.
	\end{proof}
	Let $S\subset X$ be a varifold, let $gK\in G/K=X$ and let $r\geq 0$. Put 
	\begin{equation}
		v_S(gK,r):= \int_S \chi_r(\ff(s^{-1}g))d\mathcal H^k(sK)=\int_S \chi_r(\ff(g^{-1}s))d\mathcal H^k(sK).
	\end{equation}
	The value $v_S(gK,r)$ roughly corresponds to the mass of $S$ in the $r$-ball around $gK$. Because of the smoothing effect of $\chi$ it is much more convenient to work with than $\mathcal H^k(S\cap B(gK,r)),$ which is the actual mass in the $r$-ball.
	\begin{lem}\label{lem-VarifoldMono}
		\begin{enumerate}
			\item Let $S\subset X$ be a stationary $k$-varifold. For any $gK\in X$, $r>s>0$
			$$v_S(gK,r)\geq e^{\kappa(k)\|\alpha\|(r-s)}v_S(gK,s).$$
			\item There is $\delta>0$ such that $v_S(x,0)\geq \delta$ for any stationary integral $k$-varifold $S\subset X$ and $\mathcal H^k$-almost every point $x$ of $S$.
		\end{enumerate}
	\end{lem}
	\begin{proof}
		Part (1) follows from Lemma \ref{lem-abstmonotonicity} applied to function $\mathfrak f.$ Part (2) follows from Lemma \ref{lem-smallballmono} and our choice of $\chi.$
	\end{proof}
	In \cite[Theorem 1]{anderson1982complete} Anderson proved that for any complete simply connected negatively curved manifold $N$ of sectional curvature at most $-a$, the volume of the intersection of $B(x,r)\cap S,$ where $S$ is a stationary $k$-varifold and $x$ is a regular point of $S$, is at least equal to the volume of an $r$-ball in the $k$-dimensional contractible space with constant sectional curvature $-a.$ Up to a multiplicative constant, our estimate is comparable with Anderson's for $\mathbb H^n_\mathbb R, \mathbb H^n_\mathbb C$ but it is strictly stronger for $\mathbb H^n_\mathbb H, \mathbb H^2_\mathbb O$ for $k> 4n-3, 9$ respectively. This seemingly modest improvement is at the core of our arguments in Section \ref{sec-stationaryproof}.

	\section{Monotonicity Estimates in Higher Rank}\label{sec-MonoHigherRank}
	

We first give the basic setup and an overview of the strategy.  We then specialize to the $SL(n,\mathbb{R})$ case, before giving the argument for general split groups in Section\ref{s:generalsplit}.  

Recall that for a root $\alpha$, $H_\alpha\in \fa$ is the vector satisfying $\langle H_\alpha, Y\rangle=\alpha(Y), Y\in \fa.$  Set $\rho_{\omega}=\omega (\rho)$ for $\omega$ an element of the Weyl group $W$. Then consider the function $ \mathfrak{a} \rightarrow \mathbb{R}$ defined by  $ \max_{\omega \in W} \langle \rho_\omega, x \rangle$.  Note that $f$ is invariant under the Weyl group. Our original approach was to use a smoothening of this function as a test function for the first variation formula to obtain the monotonicity formula that we need.   It will turn out that a slightly different choice of test function will work better for the final argument.  We set $\lambda = \rho - \frac{t}{2} \theta$ for $t \in (0,1)$ to be specified, and we let $f(x)= \max_{\omega \in W} \langle \lambda _\omega, x \rangle$ for $\lambda_\omega = \omega(\lambda)$.

Recall the Hessian formula in Proposition \ref{prop-CartanHess}.	We will use this formula with $f_{\epsilon}$ the Weyl group invariant $C^2$ function on $\mathfrak{a}$, where $f_{\epsilon}$ are smooth approximations to $f$ that we define now.  We will then set $F_{\epsilon}(gK):=f_{\epsilon}(a(g)).$   There is a standard way to approximate piecewise affine functions by smooth convex functions $f_{\epsilon}$, which in the case of our function $f$ is given by the following formula \cite{BoydVandenberghe2004}[\S 3.1.5]
	\begin{equation} \label{eq-softmaxf}
		f_{\epsilon}(x):= \epsilon \log \left( \sum_{\omega \in W} e^{\langle \lambda_\omega,x \rangle /\epsilon} \right) 
	\end{equation}

   \noindent  We list important properties of $f_{\epsilon}$ in the next proposition.  
	\begin{prop} \label{properties} 

	\begin{enumerate} 
    \item $f_{\epsilon}$ uniformly converges to $f$ on compact sets in the $C^0$ norm as $\epsilon \to 0$. 
		\item $f_{\epsilon}$ smoothly  converges to $f$ at the points where $f$ is smooth as $\epsilon \to 0$. 
		
		\item $f_{\epsilon}$ is invariant under the action of the Weyl group, and so gives a function $F_{\epsilon}$ on $X$ by evaluating on the Cartan coordinate. Since $f_{\epsilon}$ is smooth and $W$-invariant, $F_{\epsilon}$ is also smooth \cite{dadok1982c}.
		
		\item The gradient of $f_{\epsilon}$ is a convex combination of the $\lambda_\omega$: 
		
		\[
		\nabla f_{\epsilon}(x)= \sum_{\omega \in W}  a_\omega(x) \lambda_\omega,
		\] 
		where 
		\[
		a_\omega(x):=\frac{e^{\langle \lambda_\omega,x \rangle/\epsilon } }{\sum_{\omega \in W} e^{\langle \lambda_{\omega},x \rangle /\epsilon}}. 
		\]
		In particular, the gradient has norm at most $||\lambda||$. 
		
		\item\label{eq-softmaxfd2} The Hessian of $f_{\epsilon}$ is given by the following formula: 
		
		\[
		\nabla^2 f_{\epsilon}(x) (v,w) = \frac{1}{2\epsilon} \sum_{\omega_1,\omega_2 \in W} a_{\omega_1}(x) a_{\omega_2}(x) ((\lambda_{\omega_1} - \lambda_{\omega_2}) \cdot v) ((\lambda_{\omega_1} - \lambda_{\omega_2}) \cdot w). 
		\]

        \noindent In particular, $f_{\epsilon}$ is convex.  
		
	\end{enumerate} 
	\end{prop}
	
    We point out that the gradient of $f_{\epsilon}$ is the expectation of $\lambda_{\omega}$ with respect to the probability weights $a_{\omega}(x)$, and that the Hessian of $f_{\epsilon}$ is $1/\epsilon$ times the covariance matrix of the $\lambda_{\omega}$ with respect to the probability weights $a_{\omega}(x)$.  One can go from the formula for the Hessian in \cite[Appendix A.4]{BoydVandenberghe2004} to the formula written above by using the standard identity expressing the covariance in terms of pairwise differences. 
    
    \vspace{1mm}
    \noindent \textbf{Overview of the Strategy}
	\vspace{1mm}
	
	Fix a codimension $d$ for our minimal submanifold $\Sigma$.  We are going to show that for $n$ large enough, we have the desired monotonicity formula. This will be based on the Hessian estimate in Proposition \ref{prop-hessianlowerboundhigherrank} below. 
	
	We will obtain a uniform pointwise lower bound on the integral of the trace of the Hessian of the $F_{\epsilon}$ through the tangent subspaces to $\Sigma \cap B(R)$ for $\epsilon$ 
    small enough (this is Proposition \ref{prop-hessianlowerboundhigherrank} below.) An important reference point as we carry this out will be the computation (see \cite{cmw23}, \cite{cmw25}, \cite{fl24}) of the Hessian of the function $\beta: X \rightarrow \mathbb{R}$ defined by pairing $\lambda$   with a vector valued Busemann function (i.e., the $A$-coordinate for some choice of $NAK$ decomposition of $\SL(n,\mathbb{R})$ which allows us to write $X=NA$.)    If we are able to show that the trace of the Hessian of $F_{\epsilon}$ through codimension $d$ tangent subspaces can be taken to be at least that of the Hessian of $\beta$, up to a small error (in particular, depending only on $d$ and not $n$),  then it will turn out that we will be able to use $F_{\epsilon}$ to obtain a monotonicity formula good enough for our purposes.  
    
    An important feature of the arguments to follow is that, if $\epsilon$ is small, at points near many walls of the Weyl chamber the Hessian will be large in many directions, which will imply that the trace of the Hessian through any subspace of low codimension must be large.  On the other hand, near the center of the Weyl chamber the Hessian will very closely resemble the Hessian of $\beta$. The different scenarios addressed in the arguments below are combinations of these two extremes.  

    We recall that the Hessian of $\beta$ has a very similar form to (\ref{prop-CartanHess}) for the $C^2$ function on $\mathfrak{a}^{++}$ given by $\langle \lambda, x \rangle$.     The  Hessian matrix of $\beta$ can be written in an orthonormal basis--analogous to $\{ U_i, V_{\alpha} \}$ for $\nabla^2 F_{\epsilon}$-- relative to which it is diagonal, and so that $n-1$ upper left diagonal entries (corresponding to the $U_i)$ vanish  while the lower right diagonal entries (corresponding to the $V_{\alpha}$) are given by pairing $\alpha$ with $\lambda$. The trace of the Hessian of $\beta$ is therefore equal to $2\langle \rho, \lambda \rangle$.

    \subsection{$\SL(n,\mathbb{R})$ }
    
   We now specialize to the $\SL(n,\mathbb{R})$ case, which we include separately because some of the arguments are simpler than in the general case.  Since we give full arguments in the general case, we omit some proofs in the $SL(n,\mathbb{R})$ case.  The next lemma will be important in what follows.  
\begin{lem} \label{lem-gradientisinthepositiveweylchamber}
    For $x \in \mathfrak{a}^+$, $\nabla f_{\epsilon}(x)$ is contained in $\mathfrak{a}^+$. Equivalently, for $x \in \mathfrak{a}^+$ any positive root evaluated on $\nabla f_{\epsilon}(x)$ is non-negative.  
	\end{lem}

We prove this in the general case of split groups in Section \ref{s:generalsplit} below for $\lambda = \rho$, but the same argument works for the $\lambda$ that we consider.   Recall that the half sum $\rho$ of the positive roots satisfies:

    \begin{equation} \label{eqn:formulaforrho}
    2 \rho = (n-1)e_1 + (n-3)e_2 + ...+ (1-n) e_n
    \end{equation}
and that $ \rho - \frac{t}{2}\theta= \lambda = (\lambda_1 ,.., \lambda_n)$ satisfies $\lambda_i -\lambda_{i+1} \geq 1/2$ and $\lambda_i = -\lambda_{n+1-i}$.

   Denote by $s_i$ the reflection in the wall $W_i$. Labeling the walls of the Weyl chamber $W_1,..,W_{n-1}$, we set $a_{W_i}=a_\omega$ for $\omega= s_i$.


    

 \begin{lem} \label{largehessian}
For $\epsilon$ small enough we can choose $\delta=\delta(\epsilon)>0$ so that the following holds.  If $a_{W_i}(x)>\delta(\epsilon)$ for at least one $i$, then $\nabla^2 f_{\epsilon}(x) (v,v)>1000||\rho||^2$ for any unit vector $v$ in the span of the $e_{i}- e_{i+1}$ for $a_{W_i}(x)>\delta(\epsilon)$.  For example, we can take $\delta(\epsilon)= \epsilon^{1/2}$.
 \end{lem}

	\noindent  This is large enough to singlehandedly make the trace of the Hessian as large as we need through any codimension d tangent space to $x$ containing a vector $v$ as above: that is, larger than the trace of the Hessian of $\beta$ through any codimension $d$ subspace, where we are using that the Hessian of $F_{\epsilon}$ is non-negative definite by Lemma \ref{lem-gradientisinthepositiveweylchamber}.  

     Let us summarize the role of $\delta(\epsilon)$ in the arguments to follow: it is simultaneously large enough to make the Hessian of $f_{\epsilon}$ extremely large in certain directions if $x$ is close to walls of the Weyl chamber, and small enough that if $a_\omega(x)<\delta(\epsilon)$, then $a_\omega(x) \lambda_\omega$ will have a negligible contribution to the gradient of $f_{\epsilon}$. In what follows we will sometimes abbreviate $\delta(\epsilon)$ as $\delta$.  


     \begin{proof}
    
    That we can choose $\delta(\epsilon)$ in this way follows from the formula for the Hessian in Proposition \ref{properties}, which we can make large by making the term $\frac{1}{2\epsilon} a_{\omega_1}(x) a_{\omega_2}(x) ((\lambda_{\omega_1} - \lambda_{\omega_2}) \cdot v) ((\lambda_{\omega_1} - \lambda_{\omega_2}) \cdot w)$ large for $\omega_1=e$ (recall that $\lambda_{\omega_1}=\lambda$ so that $a_{\omega_1}(x)$ is always uniformly bounded away from zero.)  We are also using the following linear algebra fact: for any basis $\{w_1,..,w_n\}$ for a finite dimensional inner product space (in our case the $w_j$ will be some subset of the $\alpha_{i,i+1}=e_{i}-e_{i+1}$ and the inner product space will be their span), there is some $\epsilon'$ depending on the basis so that any unit vector $v$ in the inner product space satisfies $|\langle w_i,v \rangle|>\epsilon'$ for some $i$. We moreover take $\delta=\delta(\epsilon)$ tending to zero as $\epsilon \to 0$, but note that it will be the case that $\delta(\epsilon)>>\epsilon$ as $\epsilon \to 0$: for example, this is the case for $\delta(\epsilon)=\epsilon^{1/2}$. 
    
\end{proof}



 \begin{prop} \label{prop-hessianlowerboundhigherrank}
Given any $\eta>0$, provided $\epsilon$ is taken small enough we can ensure that the trace of $\nabla^2 F_{\epsilon}$ through any codimension $d$ tangent subspace is at least $2\langle \rho, \lambda \rangle - d (n-1-\frac{t}{2}) - \eta $.  
 \end{prop}

\begin{rem}
We point out that the $\epsilon$ above does not depend on $r$, although when applying Proposition \ref{prop-hessianlowerboundhigherrank} further down we will often need to choose $\epsilon$ depending on $r$.  
\end{rem}
 
\begin{proof}
  Suppose that a point $x \in \mathfrak{a}^{+}$ satisfies $a_{W_i}(x)>\delta(\epsilon)$ for more than $d $ walls $W_i$ of the Weyl chamber.  Then by Lemma \ref{largehessian} any codimension $d$ tangent subspace to $x$ will necessarily contain a vector $v$ as in the statement of that lemma. Therefore the trace of the Hessian of $F_{\epsilon}$ through such a subspace will be large enough to satisfy the conclusion of the proposition. We are using here that the Hessian of $F_{\epsilon}$ is non-negative definite (Lemma \ref{lem-gradientisinthepositiveweylchamber}), so that the large contribution of the Hessian at $v$ gives a lower bound for the whole trace through a subspace containing it.


 Next suppose that a point $x \in \mathfrak{a}^{+}$ has $a_{W_{i_\ell}}(x)>\delta(\epsilon)$ for exactly $k\leq d $ walls $W_{i_1}$,.., $W_{i_k}$ of the Weyl chamber. We need to control the diagonal terms in the Hessian corresponding to the $V_{\alpha}$ in (\ref{prop-CartanHess}.) 

It will be useful in Cases 2 and 3 below to arrange the $\{i_1, i_2, ..,i_k \}$ into blocks of consecutive numbers $B_1',B_2',..$ each of the form $i_{\ell},i_{\ell}+1,..,i_{\ell'}-1, i_{\ell'}$, and that are maximal with the property that they consist of consecutive elements of $\{i_1, i_2, ..,i_k \}$. Then we set $B_i$ equal to the set containing $B_i'$ together with the largest element of $B_i'$ plus one. We also refer to the $B_i$ as blocks. For example, if $\{i_1, i_2, ..,i_k \}$ were equal to $\{1,2,4,5,6,8\}$, the blocks $B_i$ would be $\{1,2,3\}$, $\{4,5,6,7\}$, and $\{8,9\}$.  

The following observation, which we write as a lemma, will be useful.

\begin{lem} \label{lem:cycles}

Suppose that $a_\omega(x)>\delta(\epsilon)$.  Then each cycle of the cycle decomposition of the corresponding element $\omega$ of the Weyl group is supported in one of the blocks $B_i$. 
\end{lem}

We omit the proof which for general simple Lie groups and $\lambda= \rho$ is given in the proof of Lemma \ref{lem-supprtWeyl} below. The same argument works for general $\lambda$.   

 To prove Proposition \ref{prop-hessianlowerboundhigherrank}, we need to understand the diagonal terms of the Hessian for the vectors $V_{\alpha}$ and their contributions to the trace of the Hessian.  To do so we will make reference to the Hessian of the function $\beta$ described above, whose diagonal entry corresponding to $\alpha_{i,j}$ is given by pairing $H_{\alpha_{i,j}}$ with $\lambda$.  We divide into three cases based on the root $\alpha=\alpha_{i,j}:= e_i-e_j$, $i<j$. We will find lower bounds $L_{i,j}$ (depending on $x$) for $ \nabla^2 F_{\epsilon}(V_{\alpha_{i,j}},V_{\alpha_{i,j}})$ in each of these cases.  

 \begin{itemize}
\item \textbf{Case 1}: $\alpha=\alpha_{i,j}$ for $i,j \notin \cup_{\ell=1}^k  (i_{\ell} \cup(i_{\ell}+1))$.

In this case, note first that $\frac{\cosh(\alpha(H))}{\sinh(\alpha(H))}$ will be greater than or equal to 1.   Note also that the inner product of $H_{\alpha_{i,j}}$ with $\nabla f_{\epsilon}$ will be approximately equal to the inner product of $H_{\alpha_{i,j}}$ with $\lambda$,  up to an error that can be taken as small as desired by making $\epsilon$ and $\delta(\epsilon)$ small enough, and using Lemma \ref{lem:cycles}. Therefore we can make the lower bound $L_{i,j}$ for the entry of the Hessian corresponding to $V_{\alpha_{i,j}}$ at least as large as the corresponding entry $\lambda_i-\lambda_j$ for the Hessian of $\beta$ in this case, up to an error that can be made as small as desired by making $\epsilon$ small enough.

\item \textbf{Case 2}: $\alpha=\alpha_{i,j}$ for $i,j \in \cup_{\ell=1}^k  (i_{\ell} \cup(i_{\ell}+1))$, and $i$ and $j$ contained in the same block $B_m$.  

\vspace{2mm}
We first handle the case $j=i+1$. Denote by $s_i$ the permutation swapping $i$ and $i+1$.  For $i, i+1 \in B_m$, we have that 
\[
a_{s_i}=a_e e^{- (\lambda_i-\lambda_{i+1})(x_i-x_{i+1})/\epsilon}, 
\]
which using $a_{s_i}(x) > \delta$ and $a_e(x) \leq 1$ implies 
\begin{equation} \label{ineqb}
x_{i}-x_{i+1} \leq \frac{\epsilon}{\lambda_i-\lambda_{i+1}}  \log (1/\delta).
\end{equation}

Set $\varsigma= (\varsigma_1,..,\varsigma_n) = \nabla f_{\epsilon}= \sum_{w} a_w w \lambda$. We break the permutations in $W$ into pairs $w, s_i w$, and note that each pair contributes non-negatively to $\varsigma_i - \varsigma_{i+1}$.  Using the pair containing the identity to give a lower bound for the latter, we get: 
\begin{equation} \label{ineq:1175}
\varsigma_i - \varsigma_{i+1} \geq (\lambda_i - \lambda_{i+1})(a_e-a_{s_i}) \geq \frac{(\lambda_i - \lambda_{i+1})}{n!} \left(1 - e^{\frac{-(\lambda_i - \lambda_{i+1})(x_i-x_{i+1)}}{\epsilon}}\right),
\end{equation}
where we use that $a_e \geq 1/n!$ by virtue of being at least as large as any other $a_{\omega}$.  Now note that the function $h(u)=1-e^{-(\lambda_i-\lambda_{i+1}) u / \epsilon}$ is concave, and so setting the left and right sides of (\ref{ineqb}) equal to $u_i$ and $U_i$ respectively, we get that since $0 \leq u_i \leq U_i$,
\[
h(u_i) \geq \left(1- \frac{u_i}{U_i}\right) h(0) + \frac{u_i}{U_i} h(U_i), 
\]
so because $h(0)=0$ and $h(U_i)=1- e^{-\log(1/\delta)}=1-\delta$, we have that 
\[
1-e^{-(\lambda_{i}-\lambda_{i+1})u_i/\epsilon} \geq \frac{(\lambda_{i}-\lambda_{i+1})(1-\delta)}{\epsilon \log (1/\delta)}u_i. 
\]
Substituting this into (\ref{ineq:1175}), we get 
\begin{equation} \label{ineq:iiplusone}
\varsigma_i - \varsigma_{i+1} \geq \frac{((\lambda_{i}-\lambda_{i+1})^2 (1-\delta) }{n! \epsilon \log (1/\delta)} u_i \geq \frac{ (1-\delta) }{4 n! \epsilon \log (1/\delta)} u_i=C_{\epsilon} u_i
\end{equation}
\noindent where in the last inequality we used $\lambda_{i}-\lambda_{i+1} \geq 1/2$.  We note that since $\delta = \sqrt{\epsilon}$, $C_\epsilon \to \infty$ as $\epsilon \to 0$.   If $i<j$ belong to the same block, we can sum (\ref{ineq:iiplusone}) over the range of values between them to get 
\[
\varsigma_i - \varsigma_{j} \geq C_{\epsilon} (x_i - x_j).
\]

\noindent We therefore have that 
\[
\nabla^2 F_{\epsilon}(V_{\alpha_{i,j}},V_{\alpha_{i,j}})= \coth(x_i-x_j)(\varsigma_i-\varsigma_j) \geq C_\epsilon (x_i-x_j) \coth(x_i-x_j) \geq C_{\epsilon}.  
\]
\noindent We can then choose $\epsilon$ small enough that $C_{\epsilon}>n$, and set $L_{i,j}=\lambda_i-\lambda_j <n.$

\item \textbf{Case 3}: $\alpha=\alpha_{i,j}$ for either exactly one of $i,j$  contained in $\cup_{\ell=1}^k (i_{\ell} \cup(i_{\ell}+1))$, or else both of $i,j$ contained in $\cup_{\ell=1}^k (i_{\ell} \cup(i_{\ell}+1))$ but in separate blocks $B_{m(i)},B_{m(j)}$.  

\vspace{2mm}
We emphasize that the various restrictions we impose on the size of  $\epsilon$ depend only on $n$ and $d$, but not $x$.  Let's assume first that $i$ but not $j$ is contained in $\cup_{\ell=1}^k  (i_{\ell} \cup(i_{\ell}+1))$ (the case where the reverse is true is similar.)  



We set $L_{i,j}$ equal to the pairing of $\nabla f_{\epsilon}$ with $H_{\alpha_{i,j}}$.  It is possible that $L_{i,j}$ is greater than the pairing of $\lambda$ with $H_{\alpha_{i,j}}$ or less than it.  We will check that on average these two possibilities balance out: that the sum of the $L_{i,j}$ for $\alpha_{i,j}$ in this case is at least the corresponding sum of the $\langle \lambda, H_{\alpha_{i,j}} \rangle$.



Suppose that $i$ varies within some $B_{m(i)}=\{i_{\ell},i_{\ell}+1,..,i_{\ell'}, i_{\ell'}+1\}$. Then by the formula for $\nabla f_{\epsilon}$ the coefficients of $e_{i_{\ell}},..,e_{i_{\ell'}}, e_{i_{\ell'}+1}$ in $\nabla f_{\epsilon}$ will have the same sum as the corresponding coefficients in $\lambda$, up to an error that can be made arbitrarily small by making $\epsilon$ small enough. Here as above we are using Lemma \ref{lem:cycles} to ensure that the contributions from the $a_\omega(x)\lambda_{\omega}$   for $\omega$ with support not entirely contained in the $B_i$ can be made negligible by making $\epsilon$ small.  Therefore the sum of the $L_{i,j}$ with $i$ varying in $\{i_{\ell},i_{\ell}+1,.., i_{\ell', i_{\ell'}+1}\}$ is at least as large as the sum of the corresponding entries of the Hessian of $\beta$, up to an error that can be taken as small as desired by making $\epsilon$ small.



Repeating this reasoning for each such $B_i$, and arguing similarly for when $j$ rather than $i$ is contained in $\cup_{\ell=1}^k  (i_{\ell} \cup(i_{\ell}+1))$, we conclude that the sum of the $L_{i,j}$ in each of these cases will be at least the sum of the corresponding entries of the Hessian of $\beta$, up to an error that can be made arbitrarily small by making $\epsilon$ small enough, again using Lemma \ref{lem:cycles}.   

The case of $i$ and $j$ contained in $\cup_{\ell=1}^k (i_{\ell} \cup(i_{\ell}+1))$ but in separate blocks  is similar. The sum of the $L_{i,j}$  with $i$ varying in $B_{m(i)}$ and $j$ varying in $B_{m(j)}$ will be the same as for the Hessian of $\beta$, up to an error that can be made arbitrarily small by making $\epsilon$ small enough. One checks this by looking separately at the contributions to the trace from the $e_i$ and $-e_j$ summands of $H_{\alpha_{i,j}}$ as $i$ and $j$ vary (grouping the former contributions by fixing $j$ and letting $i$ vary, and the latter contributions by fixing $i$ and letting $j$ vary.)




 \end{itemize}

   We can now finish the proof of  Proposition \ref{prop-hessianlowerboundhigherrank}.  Recall that we have a point $x$ with $a_{W_i}(x)>\delta(\epsilon)$ for at most $d$ walls $W_i$. We showed that in cases 1, 2, and 3 the sum of the $L_{i,j}$ is at least the sum of the corresponding entries of the Hessian of $\beta$, up to an error that can be made as small as desired by making $\epsilon$ small.  

  We need to find a lower bound on the trace of $\nabla^2 F_{\epsilon}$ through codimension $d$ subspaces.  The diagonal entries of $\nabla^2 F_{\epsilon}$ corresponding to the $V_{\alpha_{i,j}}$ are at least as large as the $L_{i,j}$. For entries of $\nabla^2 F_{\epsilon}$ in the upper left block, corresponding to the $U_i$ (i.e., tangent vectors in the Cartan coordinate direction), we note that $\nabla^2 F_{\epsilon}$ evaluated on a unit tangent vector in the span of the $U_i$  will be non-negative (and even potentially very large, although we do not use that in this part of the proof), and thus greater than  the corresponding entry of $\nabla^2 \beta$, which is equal to 0.  We are using here that because $f_{\epsilon}$ is convex, its Hessian is non-negative definite.  

  It follows that the trace of $\nabla^2 F_{\epsilon}$ through a codimension $d$ subspace is at least the sum of all of the $L_{i,j}$  minus the sum of the $d$ largest $L_{i,j}$.  Up to errors that can be made as small as desired by making $\epsilon$ small (i.e., up to $o(1) errors)$, the former  is at least $2 \langle \rho, \lambda \rangle $, while the latter is at most $d( n-1 - \frac{t}{2} )$.  For the latter we are using that, up to an $o(1)$ error, the $L_{i,j}$ are bounded above by $n-1-t/2$.  Depending on which case the $L_{i,j}$ falls into this follows either from the fact that the entries of the Hessian of $\beta$ are bounded above by $n-1-t/2$ (cases 1 and 2) up to an $o(1)$ error, or the fact that since $\nabla f_{\epsilon}$ is a convex combination of Weyl-group orbits of $\lambda$ (case 3), we know that its coordinates relative to the $e_i$ range in $[-(n-1)/2 + t/4,(n-1)/2 - t/4]$, and so its inner product with any $H_{\alpha_{i,j}}$ is bounded above by $n-1-t/2$.   This completes the proof.

  \end{proof}

   Let $S\subset X$ be a varifold, let $gK\in G/K=X$ and let $r\geq 0$. We set
	\begin{equation}
		v_S(gK,r,\epsilon):= \int_S \chi_r\left(\frac{F_{\epsilon}(s^{-1}g)}{||\lambda||}\right)d\mathcal H^k(sK)=\int_S \chi_r\left(\frac{F_{\epsilon}(g^{-1}s)}{||\lambda||}\right)d\mathcal H^k(sK), 
	\end{equation}  
    where to obtain the last equality we have used $F_{\epsilon}(h)=F_{\epsilon}(h^{-1})$, which follows from the fact that $\lambda_i=-\lambda_{n+1-i}$ observed above.  
    We have the following version of Lemma \ref{lem-VarifoldMono} in the higher rank setting.  
	\begin{lem}\label{lem-VarifoldMonohigherrank}
		\begin{enumerate}
			\item Let $S\subset X$ be a stationary varifold of codimension $d$. For any $x\in X$, and $\eta>0$, there is $\epsilon'>0$ so that for all positive $\epsilon<\epsilon'$ and $r>s>0$ we have that
			$$v_S(gK,r,\epsilon)\geq e^{ \left(\frac{2\langle \rho, \lambda \rangle - d (n-1-\frac{t}{2}) - \eta}{||\lambda||} \right)(r-s)} v_S(gK,s,\epsilon).$$
			\item There is $\delta>0$ such that for all $\epsilon$ sufficiently small $v_S(gK,0,\epsilon)\geq \delta$, for any stationary integral $k$-varifold $S\subset X$ and $\mathcal H^k$-almost every point of $S$.
		\end{enumerate}
	\end{lem}
\begin{proof}
  By Proposition \ref{prop-hessianlowerboundhigherrank}, the trace of the Hessian of $F_{\epsilon}/||\lambda||$ through any codimension $d$ tangent subspace is at least
\[
\frac{1}{||\lambda||} \left(2\langle \rho, \lambda \rangle - d (n-1-\frac{t}{2}) - \eta  \right), 
\]
\noindent provided $\epsilon$ was chosen sufficiently small. Since $F_{\epsilon}/||\lambda||$ has gradient bounded above in norm by one by Proposition \ref{properties}, we can apply Lemma \ref{lem-abstmonotonicity} to conclude the first part. The second part follows from Lemma \ref{lem-smallballmono} as in Lemma \ref{lem-VarifoldMono}.

    \end{proof}

    \subsection{General Split Simple Groups} \label{s:generalsplit}

    We describe the modifications necessary to extend the argument beyond $\SL_n(\mathbb{R})$ to the case of general split simple non-compact real groups $G$. 
    We keep the notations from before. One part of the argument that simplifies in this setting is that since we don't try to optimize the leading constant in the general case of Theorem \ref{mt-higherrank}, we are able work with $\rho$ rather than $\lambda$.

    Let $r$ be the rank of $G$ and let $\Delta$ be the set of simple roots in $\Phi^+$. For any root $\alpha\in \Phi$ let $s_\alpha$ be the reflection in $\alpha$: 
    $$s_\alpha(x):=x-\alpha(x)\check\alpha, \check\alpha:=\frac{2H_\alpha}{\|\alpha\|^2}$$
    We recall that the Weyl group is generated by the reflections in the simple roots, which we call simple reflections. The minimal number of simple reflections needed to write down $w\in W$ is the \emph{length function} $\ell(w)$. The set of simple reflections needed to write down $w\in W$ is called the support of $w$, denoted $\supp(w)$. It doesn't depend on the minimal length expression chosen to represent $w$, which in general is not unique. The support has a nice geometric interpretation we shall use in our proofs. 
    \begin{lem}\label{lem-suppgeo}
    $$  
    \supp w=\{s_\alpha|\alpha\in \Delta, \frak a^+\cap w\frak a^+\subset \ker \alpha\}.
    $$
    \end{lem}
    In other words, the support of $w$ consists of reflections in those simple roots whose corresponding faces contain the intersection of the Weyl chamber $\frak a^+$ with its image by $w$. 
    \begin{proof}Since $W$ acts on $\frak a$ with a strict fundamental domain $\frak a^+$ \cite[Thm. 11.2(c)]{humphreys1990reflection}, the intersection $F:=\frak a^+\cap w\frak a^+$ must be pointwise fixed by $w$, moreover it has to be the intersection of $\frak a^+$ with $\frak a^w:=\{X\in \frak a | wX=X\}$. By the point stabilizer theorem \cite[Theorem~1.12(a)]{humphreys1990reflection}, the pointwise stabilizer of $F$ is generated by the reflections in the simple roots $\alpha$ vanishing on $F$. In particular, $w$ can be generated using only reflections in such $\alpha$, hence $\supp w\subset \{s_\alpha|\alpha\in \Delta, \frak a^+\cap w\frak a^+\subset \ker \alpha\}$. For the reverse inclusion, we note that  $w$ can be written as a product of $s_{\alpha},\alpha\in \supp w$ so $\frak a^+\cap \bigcap_{\alpha\in\supp w}\ker\alpha$ is a subset of $\frak a^+$ pointwise fixed by $w$. It is therefore contained in  $\frak a^+\cap \frak a^w=\frak a^+\cap w\frak a^+\subset \cap_{\frak a^+\cap w\frak a^+\subset \ker\alpha}\ker\alpha$. We deduce $\{s_\alpha|\alpha\in \Delta, \frak a^+\cap w\frak a^+\subset \ker \alpha\}\subset \supp w$.
    \end{proof}

    Just like in the $\SL_n(\mathbb R)$ case, we use the function $f_\varepsilon$ given by the formula (\ref{eq-softmaxf}). We choose an auxiliary parameter $\delta=\varepsilon^{1/2}>0$. The power $1/2$ is not important, as the argument would work the same for any fixed positive power of $\varepsilon$ less than one. For $x\in\frak a^+$ let $\Delta_x$ denote the set of simple roots $\alpha$ such that $a_{s_\alpha}(x)\geq \delta.$ Let $\frak a_x^*$ be the subspace of $\frak a^*$ spanned by $\alpha\in\Delta_x$, similarly let $\frak a_x$ be the dual of $\frak a_x^*.$ We define $W_x$ as the subgroup of the Weyl group generated by the simple reflections $s_\alpha,\alpha\in \Delta_x.$ Finally define $\Phi^+_x=\Phi^+\cap \frak a_x^*$.

    Lemma \ref{lem-gradientisinthepositiveweylchamber} holds for simple Lie groups in general, and we now give the proof in the general case. 
    \begin{proof}[Proof of Lemma \ref{lem-gradientisinthepositiveweylchamber} in general case]

    Formula \ref{properties} (\ref{eq-softmaxfd2}) for the second derivative implies that $D^2 f_\varepsilon(x)(H_\alpha,H_\alpha)\geq 0$ for any simple root $\alpha\in \Delta$ and $x\in \frak a$. As $x\in\frak a^+$, there is a $t_0\geq 0$ such that $\alpha(x-t_0 H_\alpha)=0$. By the symmetry of $f_\varepsilon$ under the reflection $s_\alpha$, $\langle \nabla f_\varepsilon(x-t_0H_\alpha),H_\alpha\rangle =0,$ hence $$\langle \nabla f_\varepsilon(x),H_\alpha\rangle =\int_0^{t_0} D^2 f_\varepsilon(x+(t-t_0)H_\alpha)(H_\alpha,H_\alpha)dt\geq 0,$$ as desired.
    \end{proof}

    Likewise, Lemma \ref{largehessian} extends to the general split case. 
    \begin{lem}\label{lem-largehessiangen}
    For $\varepsilon$ small enough the following holds. Let $x\in \frak a^+$. Any vector $v\in \frak a_x$ satisfies 
    $$D^2 f_\varepsilon(x)(v,v)\geq 1000 \|\rho\|^2\|v\|^2.$$
    \end{lem}
    \begin{proof}
    Recall that $\delta=\varepsilon^{1/2}$ and let $v\in \frak a$ be a unit vector in $\frak a_x$. We note that $\rho-s_\alpha\rho= \alpha, \alpha\in \Delta$ and that $a_1(x)\geq 1/|W|.$ It follows
    \begin{align*}
    \nabla^2 f_\varepsilon(x)(v,v)=&\frac{1}{2\epsilon}\sum_{w,w'\in W}a_w(x)a_{w'}(x)(w\rho-w'\rho)(v)^2\\
    \geq& \frac{1}{2\varepsilon}\sum_{\alpha\in \Delta_x}a_1(x)a_{s_\alpha}(x)(\rho-s_\alpha \rho)(v)^2\\
    \geq& \frac{1}{2|W|\varepsilon}\varepsilon^{1/2}\sum_{\alpha\in \Delta_x} \alpha(v)^2\geq \frac{\varepsilon^{-1/2}}{2|W|C}\|v\|^2,
    \end{align*}
    where $C$ is the constant from the norm comparison theorem applied to the space $\frak a_x$. By taking $\varepsilon$ small enough we can ensure that  $\varepsilon^{-1/2}/(2|W|C)\geq 1000 \|\rho\|^2.$
    \end{proof}

The following Lemma generalizes Lemma \ref{lem:cycles} (or rather its proof does, which also works for $\rho$ replaced by $\lambda$ from the previous section.) Here the proof gets slightly more complicated, as we are allowed only to use general facts about root systems.
    \begin{lem}\label{lem-supprtWeyl}
    Let $x\in\frak a^+$. Suppose $w\in W$ is such that $a_w(x)>\delta$. Then $ \supp w \subset \Delta_x$. 
    \end{lem}
    \begin{proof}
    We will first show that 
    \begin{equation}\label{eq-suppinequality} \langle s_\alpha y -wy,x\rangle \geq 0
    \end{equation}for every $x,y\in \frak a^+$ and $\alpha\in \supp w.$ The lemma will then follow relatively quickly.

    Inequality (\ref{eq-suppinequality}) is bi-linear in $x,y$, so it is enough to prove it for $x,y$ in one dimensional faces of $\frak a^+$, as every other $x,y\in \frak a^+$ is a convex combination of such vectors. Each of these faces is spanned by a fundamental co-weight $\check{w}_\beta, \beta\in \Delta$. We recall that $\eta(\check{w}_\beta)=\delta_{\eta\beta},\eta\in \Delta,$ so the one dimensional face corresponding to $\check{w}_\beta$ is the intersection of all the faces other than $\ker \beta\cap \frak a^+.$  We have thus reduced checking (\ref{eq-suppinequality}) to the inequality 
    \begin{equation}\label{eq-fundweightineq}\langle s_\alpha \check w_\beta -w\check w_\beta,\check w_\gamma\rangle \geq 0,\end{equation} for all $\beta,\gamma\in \Delta.$ 
    We record the identity $s_\beta (\check w_\gamma)=\check{w}_\gamma - \delta_{\gamma\beta}\check\beta$, $\check\beta=2H_\beta/\|\beta\|^2$
    , which shall be used frequently in the proof.

    In case $\beta\neq\alpha,$ we have $s_\alpha(\check{w}_\beta)=\check{w}_\beta$ so 
    $$\langle s_\alpha \check w_\beta -w\check w_\beta,\check w_\gamma\rangle=\langle \check w_\beta -w\check w_\beta,\check w_\gamma\rangle.$$
    \textbf{Claim}  $\langle \check w_\beta, \check w_\gamma\rangle \geq \langle w\check w_\beta,\check w_\gamma\rangle$ for all $w\in W$. We proceed by induction on $\ell(w).$ Write $w=w's_\eta,\eta\in \Delta, \ell(w')=\ell(w)-1$. If $\eta\neq \beta$ then 
    $\langle w\check w_\beta,\check w_\gamma\rangle=\langle w'\check w_\beta,\check w_\gamma\rangle$ and we conclude by the inductive hypothesis. If $\eta=\beta,$ then $\langle w\check w_\beta,\check w_\gamma\rangle=\langle w'\check w_\beta - w' \check \beta,\check w_\gamma\rangle.$ By \cite[Thm 1.7]{humphreys1990reflection}, $\ell(w')=\ell(ws_\beta)<\ell(w)$ implies $w\beta\in -\Phi^+,$ hence $w' \check \beta=ws_\beta \check \beta=-w \check \beta$, which is a positive multiple of an element of $\Phi^+$. Therefore $\langle  w' \check \beta,\check w_\gamma\rangle\geq 0$ 
    and $$\langle w'\check w_\beta - w' \check \beta,\check w_\gamma\rangle\leq \langle w'\check w_\beta,\check w_\gamma\rangle.$$ We can again conclude by the inductive hypothesis. The claim is proved. 
    
    Using the claim we get $\langle \check w_\beta -w\check w_\beta,\check w_\gamma\rangle\geq 0$ which finishes the proof of (\ref{eq-fundweightineq}) in the first case. 

    In case $\gamma\neq\alpha$, we have
    $$\langle s_\alpha \check w_\beta -w\check w_\beta,\check w_\gamma\rangle=\langle \check w_\beta,s_\alpha \check w_\gamma\rangle  -\langle w\check w_\beta,\check w_\gamma\rangle=\langle \check w_\beta, \check w_\gamma\rangle  -\langle w\check w_\beta,\check w_\gamma\rangle,$$ which is handled by the Claim similarly to the previous case.
    
    Finally we consider the case $\beta=\gamma=\alpha.$  Note that $s_\alpha(\check w_\alpha)=\check w_\alpha-\check\alpha$.  We can write
    \[
    \check w_{\alpha} - w \check w_{\alpha} = \sum_{\beta \in \Delta} m_\beta \check \beta  
    \]

    \noindent for integers $m_{\beta}$.  We then have using that $w$ is orthogonal that  
    \[
    \frac{2m_{\alpha}}{||\alpha||^2}= \langle \check w_{\alpha} - w \check w_{\alpha}, \check w_{\alpha} \rangle  = \frac{1}{2} || \check w_{\alpha} - w \check w_\alpha ||^2>0.
    \]
    That $\check w_{\alpha} \neq w \check w_\alpha$ follows from the fact that otherwise $\check w_\alpha\in \frak a^+\cap w\frak a^+, $ contradicting $\frak a^+\cap w\frak a^+\subset \ker \alpha.$ We have by integrality that $m_{\alpha} \geq 1$, so putting everything together we get: 
     \[
     \langle s_{\alpha} \check w_{\alpha} - w \check w_{\alpha}, \check w_{\alpha} \rangle = \langle \check w_{\alpha}- w \check w_{\alpha} -\check\alpha, \check w_{\alpha} \rangle = \frac{2(m_{\alpha}-1)}{||\alpha||^2} \geq 0, 
     \]
    \noindent which concludes the last case of (\ref{eq-fundweightineq}).


    We have proved (\ref{eq-suppinequality}). Let us now finish the proof of the lemma. Suppose $x\in \frak a^+,w\in W$ and $a_w(x)\geq \delta.$ Let $y=H_\rho$ be the vector satisfying $\langle y,x\rangle=\rho(x).$ Then $y\in \frak a^+$ and (\ref{eq-suppinequality}) yields $s_\alpha\rho(x)\geq w\rho(x)$ for all $\alpha\in\supp w$ and $x\in \mathfrak{a}^+$. This directly implies $a_{s_\alpha}(x)\geq a_w(x)\geq \delta, \alpha\in\supp w,$ so $\supp w\subset \Delta_x$
    \end{proof}

    \begin{lem} \label{lem:1046}
    For $\varepsilon$ small enough the following holds. Let $x\in \frak a^+,k\in K$ and let $\alpha\in \Phi_x^+$
    $$D^2 F_\varepsilon(ke^x)(V_{\alpha,i},V_{\alpha,i})\geq  1000 \|\rho\|^2,\, i=1,\ldots,m_\alpha$$
    \end{lem}
    \begin{proof}

The argument is similar to the argument in case 2 of the proof of Proposition \ref{prop-hessianlowerboundhigherrank}, from which we keep the notation that was introduced. Recall that we set 
\[
\varsigma = \varsigma(x) = \nabla f_{\epsilon}(x).  
\]

For each simple root $\beta$, set $q_{\beta}= \langle \rho, \check{\beta} \rangle = 2 \frac{\langle \rho, \beta \rangle }{||\beta||^2}$. Note that $\rho - s_{\beta} \rho = q_{\beta} \beta$, and that $a_{s_\beta}(x) = a_e(x) e^{-q_{\beta} \beta(x)/\epsilon}$.  If $\beta \in \Delta_x$, then using  $a_{s_{\beta}}(x) \geq \delta=\epsilon^{1/2}$ (Lemma \ref{lem-supprtWeyl}) and $a_{e}(x) \leq 1$, we have that 
\begin{equation} \label{eq:splitbeta}
0 \leq \beta(x) \leq \frac{\epsilon}{q_{\beta}} \log (1/\delta).
\end{equation}

As in the $SL(n,\mathbb{R})$ case, we partition the Weyl group into pairs $w, s_\beta w$.  It follows from the fact that 
\[
\frac{a_{s_\beta w}}{a_w} =  e^{-\frac{ 2\langle w \rho, \beta \rangle \beta(x)}{\epsilon || \beta||^2}} 
\]
that the contribution of each pair $w, s_\beta w$, 
\[
(a_w - a_{s_\beta w}) \langle w \rho, \beta \rangle,  
\]
is nonnegative. We therefore have that 

\begin{align} \label{ineq:firstpair}
\langle \varsigma(x), \beta \rangle &\geq \langle \rho, \beta \rangle (a_e(x) - a_{s_{\beta}}(x))  \\
&\geq  \frac{\langle \rho , \beta \rangle }{|W|} \left( 1 - e^{-q_{\beta} \beta(x)/\epsilon} \right),    
\end{align}
where we have used the pair $e,s_{\beta}$ to get a lower bound for the whole inner product, and that $a_{e}(x) \geq 1/|W|$ as the largest weight.  

We use the same concavity argument as in Case 2 of Proposition \ref{prop-hessianlowerboundhigherrank}, applied this time to the function
\[
h(u)= 1 - e^{-q_{\beta} u/\epsilon}, 
\]
for which $h(0)=0$ and $h(U)=h\left(\frac{\epsilon}{q_{\beta}} \log (1/\delta)\right)= 1- \delta$.  Setting $u= \beta(x)$, by (\ref{eq:splitbeta}) we have that $0 \leq u \leq U$, so taking the lower bound on $h(u)$ from concavity of $h$ and substituting into \eqref{ineq:firstpair}, we get 
\begin{equation}
\langle \varsigma, \beta \rangle \geq \frac{ q_{\beta} \langle \rho , \beta \rangle (1-\delta)}{|W| \epsilon \log (1/\delta)} \beta(x). 
\end{equation}

\noindent We now set  
\[
C_{\epsilon}= \frac{ 1-\delta}{|W| \epsilon \log (1/\delta)} \left( \min_{\beta \in \Delta}  \frac{2 \langle \rho, \beta \rangle^2 }{||\beta||^2}\right). 
\]
and we have (recalling that $q_{\beta}= 2 \frac{\langle \rho, \beta \rangle }{||\beta||^2}$)

\begin{equation} \label{ineq:simpleceps}
\langle \varsigma(x), \beta \rangle \geq C_{\epsilon} \beta(x) 
\end{equation}
for all $\beta \in \Delta_x$.   Since $\delta = \epsilon^{1/2} $, $C_{\epsilon}$ tends to infinity as $\epsilon \to 0$.  

We now handle the non-simple elements $\alpha \in \Phi_x^+$. For such an $\alpha$, we can write its simple root expansion as: 
\begin{equation} \label{1341}
\alpha= \Sigma_{\beta \in \Delta_x} m_\beta \beta, \hspace{2mm} m_{\beta} \geq 0.  
\end{equation}
Multiplying \eqref{ineq:simpleceps} by $m_\beta$ and summing over the $\beta$ that appear in \eqref{1341}, it follows that: 
\begin{equation} \label{eqn:1345}
Df_{\epsilon} (x) (H_{\alpha}) \geq C_{\epsilon} \alpha(x) 
\end{equation}
By Proposition \ref{prop-CartanHess} and the computation from the proof of Lemma \ref{lem-gradientisinthepositiveweylchamber}, 
     \begin{equation*}D^2 F_\varepsilon(ke^x)(V_{\alpha,i}, V_{\alpha,i})=\coth(\alpha(x))Df_\varepsilon(x)(H_\alpha)\geq \frac{1}{\alpha(x)}Df_\varepsilon(x)(H_\alpha).
     \end{equation*}
Therefore by \eqref{eqn:1345}, the rightmost term above is bounded below by $C_{\epsilon}$, so choosing $\epsilon$ small enough that $C_{\epsilon}> 1000 ||\rho||^2$ concludes the proof of the Lemma.

    \end{proof}
    As before, we let $F_\varepsilon(gK):=f_\varepsilon(a(g))$.
    We can now prove the analogue of Proposition \ref{prop-hessianlowerboundhigherrank}. 
    \begin{prop}
    Let $G$ be a split noncompact simple real Lie group of rank $r$. then for any $\eta>0$, for $\varepsilon>0$ small enough the following holds true. For any $x\in X$, any subspace $S\subset T_{x} X$ of codimension $d$ we have 
    $$\tr \nabla^2 F_\varepsilon(x)\mid_S\geq 2\|\rho\|^2- d \max _{\alpha \in \Phi^+} \langle \rho, \alpha \rangle - \eta.$$
    \end{prop}
    \begin{proof}
    We shall use the vector fields ${\bf U}_H, H\in\frak a, \bf V_\alpha,\alpha\in \Phi^+$ defined in Section \ref{sec-Hess}. Recall that for each point $x\in X^{++}$ the vectors ${\bf V}_\alpha(x)$ for $\alpha\in\Phi^+$ and ${\bf U}_{H_i}(x)$ where $H_i$ is any orthonormal basis of $\frak a$, provide an orthonormal basis of $T_x X$. Let us write $U$ for the subspace spanned by ${\bf U}_H(x),H\in \frak a$, $V$ for the space spanned by ${\bf V}_\alpha(x)$, $U_x$ for the space spanned by ${\bf U}_H(x), H\in \frak a_x$ and finally $V_x$ for the space spanned by ${\bf V}_\alpha(x)$ where $\alpha$ is a non-negative combination of simple roots in $\Delta_x$. We have $\dim U_x=|\Delta_x|$.
    
    First, we consider the case $|\Delta_x|>d.$ In this case, $\dim S+\dim U_x>\dim X$ so $S$ must intersect $U_x$ non-trivially. Choose an $H\in \frak a_x, \|H\|=1$ such that ${\bf U}_H\in S.$ By Lemma \ref{lem-largehessiangen}, Lemma \ref{lem-gradientisinthepositiveweylchamber} and Proposition \ref{prop-CartanHess}
    $$ \tr \nabla^2F_\varepsilon(x)\mid_S\geq \nabla^2F_\varepsilon(x)(H,H)\geq 1000\|\rho\|^2\|H\|^2,$$ which concludes the proof of this case. 
    
    Suppose now that $|\Delta_x|\leq d.$ By Lemma \ref{lem-supprtWeyl}, any $w\in W$ such that $a_w(x)\geq \varepsilon^{1/2}$ satisfies $\supp w \subset \Delta_x$ so $w\in W_x$. 
    
   

    Write $\rho=\rho_x+\rho_x^\perp,$ where $\rho_x$ is the orthogonal projection of $\rho$ onto $\frak a_x.$ It is not hard to see that $\rho_x$ is actually the half sum of roots in $\Phi^+_x.$ Indeed, if $\beta \in \Phi^+\setminus \Phi^+_x$ then $w\beta\in \beta+\frak a_x^*$ for any $w\in W_x$. In particular, the whole $W_x$ orbit of $\beta$ stays inside $\Phi^+ - \Phi_x^+$, as the action by $W_x$ cannot change the sign of coefficients of simple roots in $\Delta\setminus\Delta_x.$ This means that   the projection of $\sum_{\beta\in \Phi^+\setminus \Phi^+_x}\beta$ onto $\frak a_x^*$ is $W_x$ invariant, hence must be zero.

    In place of the $L_{i,j}$ in the argument for $SL(n,\mathbb{R})$, we will work with quantities $L_{\alpha}, \alpha \in \Phi^+$, in order to get a lower bound on the trace of the Hessian.  We set $L_{\alpha}=\langle \rho,\alpha \rangle $ for $\alpha \in \Phi_x^+$, and $L_{\alpha}=\langle \varsigma ,\alpha \rangle $ otherwise.  
    
    We claim that each $L_{\alpha}$ is a lower bound for the corresponding diagonal entry $\nabla^2 F_{\epsilon}(V_\alpha, V_{\alpha})$  of the Hessian provided $\epsilon$ was chosen small enough. For $\alpha \in \Phi_x^+$, this follows from Lemma \ref{lem:1046} .  For $\alpha \in \Phi^+ - \Phi_x^+$ this follows from the fact that $\nabla^2F_{\epsilon} (V_{\alpha}, V_{\alpha}) \geq  \langle \varsigma , \alpha \rangle $.  

    Note also that we have the following estimate for the individual $L_{\alpha}$: 
    \begin{equation} \label{ineq-646}
    0 \leq L_{\alpha} \leq \max _{\alpha \in \Phi^+} \langle \rho, \alpha \rangle,
    \end{equation}
\noindent where the first inequality follows from Lemma \ref{lem-gradientisinthepositiveweylchamber}.  


We now estimate the sum of the $L_{\alpha}$, using for the second equality the fact observed above that $\rho_x$ is equal to the half sum of roots in $\Phi^+_x$: 
    \begin{equation} \label{ineq-1467}
    \begin{aligned}
    \sum_{\alpha \in \Phi^+} L_{\alpha} =& \sum_{\alpha \in \Phi_x^+} L_{\alpha} + \sum_{\alpha \in \Phi^+ -\Phi_x^+} \langle \varsigma, \alpha \rangle \\
    &= 2 || \rho_x||^2 + 2 \langle \varsigma, \rho_x^\perp \rangle \\
    &= 2 ||\rho_x||^2 + 2 \langle \rho, \rho_x^\perp \rangle + 2 \sum_{\omega \in W}  a_\omega \langle \omega \rho - \rho, \rho_x^\perp \rangle  \\
    &= 2 ||\rho||^2 + 2 \sum_{\omega \notin W_x} a_{\omega} \langle \omega \rho - \rho, \rho_x^\perp \rangle,
    \end{aligned} 
    \end{equation}
\noindent where in the last equality we use that for $\omega \in W_x$, 
\[
\langle \omega \rho, \rho_x^\perp \rangle = \langle  \rho,  \omega^{-1} \rho_x^\perp \rangle\ = \langle  \rho, \rho_x^\perp  \rangle. 
\]

\noindent This is the part of the argument analogous to the block sum averaging-out in Case 3 of the argument for $SL(n,\mathbb{R})$. By Lemma \ref{lem-supprtWeyl}, we have that $a_{\omega}(x) < \epsilon^{1/2}$ for $\omega \notin W_x$, and so since 
\[
| \langle \omega \rho - \rho, \rho_x^\perp \rangle| \leq 2 ||\rho||^2, 
\]
\noindent it follows from (\ref{ineq-1467}) that 
\[
\sum_{\alpha \in \Phi^+} L_{\alpha} \geq 2 ||\rho||^2 - 4 |W| ||\rho||^2 \epsilon^{1/2 }. 
\]
We choose $\epsilon$ small enough so that the last term is at most $\eta$.  Using (\ref{ineq-646}), we thus conclude the proof of the Proposition as in the argument for $SL(n,\mathbb{R})$.  

\end{proof}

    The norms of half-sums of positive roots across simple root systems can be computed via the Freudenthal–de Vries formula $\Vert{}\rho\Vert{}^2 = \frac{1}{12}h^\vee \dim(\mathfrak{g})$ \cite{freudenthal1969linear}. The explicit  values across all classical root types follow from the standard tables of Lie algebra invariants (see, e.g., \cite[Table 1]{kac1990infinite} or \cite[Ch. VI, Plates I–IX]{bourbaki2002lie}). The norm squared of the half sum of positive roots in a classical root system of rank at most $d$ is at most $d^3/3$ for $r>1$, as can be deduced from the table below. 
    \begin{table}[htpb]\label{tab-rholengths}
\centering
\renewcommand{\arraystretch}{1.8}
\begin{tabular}{l c c l}
\hline
{Root System} & {dual Coxeter} $h^\vee$ & $\dim(\mathfrak{g})$ & $\|\rho\|^2$ \\
\hline
$A_r$ ($r \ge 1$) & $r+1$ & $r(r+2)$ & $\frac{r(r+1)(r+2)}{12}$ \\
$B_r$ ($r \ge 2$) & $2r-1$ & $r(2r+1)$ & $\frac{r(4r^2 - 1)}{12}$ \\
$C_r$ ($r \ge 3$) & $r+1$ & $r(2r+1)$ & $\frac{r(r+1)(2r+1)}{12}$ \\
$D_r$ ($r \ge 4$) & $2r-2$ & $r(2r-1)$ & $\frac{r(r-1)(2r-1)}{6}$\\
\hline
\end{tabular}
\caption{Squared norm of the half-sum of positive roots under the normalization ($\|\alpha_{\text{long}}\|^2 = 2$).}
\label{tab:weyl_vector_norm}
\end{table}


	\section{Stationary Varifolds and matrix coefficients}\label{sec-VarifoldsMonoMC}
	Let $\Gamma$ be a discrete subgroup of $G$. In this section we will use the monotonicity estimates to control the decay of certain matrix coefficients of the right regular representation $L^2(\Gamma\bs G)$. The discussion in Section \ref{s:6pointone}  applies to all discrete subgroups of $G$ although the corollaries seem to be most interesting in the case of lattices. We use the tension between the monotonicity estimates and the decay of matrix coefficients to prove Theorem \ref{mt-stationary}.
	
	\subsection{Convolutions, unitary representations and spherical functions.} \label{s:6pointone}
	The discussion in this part applies to all semi-simple real Lie groups with no restrictions on the rank. Let us recall the definition of the convolution for functions in $C(G)$. For $\varphi_1,\varphi_2\in C_c(G)$ we put
	$$ \varphi_1\ast \varphi_2(h)=\int_G \varphi_2(g^{-1}h)\varphi_1(g)dg=\int_G \varphi_1(hg^{-1})\varphi_2(g)dg.$$
	The same definition applies when only one of the functions is in $C_c(G)$ and the other is only locally $L^1$-integrable or when both functions are $L^2$-integrable. In particular, if $f\in L^2(\Gamma\bs G)$, then $f\ast \varphi\in L^2(\Gamma\bs G)$ for any $\varphi\in C_c(G).$ Let us explain the interplay between convolutions and right regular unitary representations in more detail.
	Let $\lambda_{\Gamma\bs G}\colon G\to \mathscr U (L^2(\Gamma\bs G))$ be the right regular representation 
	$$\lambda_{\Gamma\bs G}(g)f(h):=f(hg), \quad f\in L^2(\Gamma\bs G).$$
	Note that $\lambda_{\Gamma\bs G}$ is a left action $\lambda_{\Gamma\bs G}(g_1)\lambda_{\Gamma\bs G}(g_2)f=\lambda_{\Gamma\bs G}(g_1g_2)f.$
	For any $\varphi\in L^1(G)$ we define 
	$$[\lambda_{\Gamma\bs G}(\varphi)f](h) := \int_{G}\varphi(g)(\lambda_{\Gamma\bs G}(g)f(h))dg, \quad f\in L^2(\Gamma\bs G), h\in 
	G.$$ The resulting function $\lambda_{\Gamma\bs G}(\varphi)f$ is left $\Gamma$-invariant and descends to a square integrable function on $\Gamma\bs G.$
	We have  
	\begin{equation}\lambda_{\Gamma\bs G}(\varphi)f(h)=\int_G\varphi(g)f(hg)dg=f\ast \check{\varphi}(h), \text{ where } \check{\varphi}(g):=\varphi(g^{-1}).\end{equation}\label{eq-convcheck}
	(Here $\varphi$ is always a real-valued function.) Convolutions play nicely with the operators defined just above
	\begin{equation}\label{eq-convrreg}\lambda_{\Gamma\bs G}(\varphi_1)\lambda_{\Gamma\bs G}(\varphi_2)f=\lambda_{\Gamma\bs G}(\varphi_1\ast \varphi_2)f, \quad \varphi_1,\varphi_2\in C_c(G).\end{equation}
	In particular $\varphi\mapsto \lambda_{\Gamma\bs G}(\varphi)$ is an algebra homomorphism from $C_c(G)$ with convolution to the algebra of bounded operators on $L^2(\Gamma\bs G).$ We also have 
	\begin{equation}\label{eq-convassoc}
		(f\ast \varphi_1)\ast \varphi_2= f\ast (\varphi_1\ast \varphi_2), \quad f\in L^2(\Gamma\bs G), \varphi_1,\varphi_2\in C_c(G). 
	\end{equation}
	\noindent and 
	\begin{equation}\label{eq-adjoint}
		\langle f_1\ast \varphi, f_2\rangle=\langle f_1, f_2\ast \check{\varphi}\rangle.
	\end{equation}
	Recall that $K$ was a fixed maximal compact subgroup of $G$. A function $f\colon G\to \mathbb C$ is called spherical if $f(k_1gk_2)=f(g)$ for all $k_1,k_2$. We write $L^p(G\sslash K), C(G\sslash K)$ for the spaces of $L^p$-integrable and continuous spherical function respectively. In the sequel we will heavily rely on the fact that convolution of spherical functions is commutative. 
	
	\begin{lem}
		Let $\varphi_1,\varphi_2\in L^1(G\sslash K)$ be spherical functions. Then $\varphi_1\ast \varphi_2=\varphi_2\ast \varphi_1.$
	\end{lem}
	\begin{proof}
		It follows from \cite[Prop. 1.5.2]{gv88} combined with the definition \cite[Def. 1.5.1]{gv88}. 
	\end{proof}
	For any $k$-varifold $S\subset \Gamma\bs X$ write $\tilde S$ for the lift to $X$. 
	$$\tilde S=\Gamma S=\{x\in X\mid \Gamma x\in S\}.$$
	Since being stationary is a local property, $S$ is stationary if and only if $\tilde S$ is.
	\begin{defn} For $\varphi$ a compactly supported spherical function the formula 
		$$(\tilde S\ast \varphi)(h):=\int_{\tilde S} \varphi(g^{-1}h)d\mathcal H^k(gK)$$ for cocompact $\Gamma$ defines a left $\Gamma$-invariant, right $K$-invariant function on $G$ which descends to a right $K$-invariant function in $L^2(\Gamma\bs G)$, denoted $S\ast \varphi$.
	\end{defn}
	
	From now on let us fix a smooth compactly supported spherical test function $\varphi\in C_c^\infty(G)$ such that $\varphi(g)\geq 0, \int_G \varphi(g)dg=1$. The constants appearing in the results below will depend on our choice of $\varphi$ but not on $S$ or $\Gamma$. 
	We put \[\mathfrak u:=S\ast \varphi\in L^2(\Gamma\bs G)\text{ and }\Psi(g):=\langle \mathfrak u,\lambda_{\Gamma\bs G}(g^{-1})\mathfrak u\rangle=\langle \mathfrak u, \lambda_{\Gamma\bs G}(g)\mathfrak u\rangle,\]
	where the last equality follows from (\ref{eq-convcheck}) and (\ref{eq-adjoint}). Choose a Borel measurable fundamental domain $S_0\subset \tilde S$ of $S$, so that $\tilde S=\bigsqcup_{\gamma\in\Gamma}\gamma S_0.$
	
	\begin{lem}\label{lem-BallIntForm} 
    
    For any  smooth compactly supported spherical function $\Phi\colon G\to\mathbb R$
		$$\int_{G}\Phi(g) \Psi(g) dg= \int_{S_0}\int_G \left(\int_{\tilde S}\Phi(s^{-1}uh)d\mathcal H^k(sK)\right)(\varphi\ast\check{\varphi})(h)dh d\mathcal H^k(uK). $$
	\end{lem}

	\begin{proof}
		By the bilinearity of the inner product we have that the left side equals  
		$$\int_{G}\Phi(g)\langle S\ast \varphi, \lambda_{\Gamma\bs G}(g^{-1})(S\ast \varphi)\rangle dg=\langle S\ast \varphi, \lambda_{\Gamma\bs G}(\check{\Phi})(S\ast \varphi) \rangle=\langle S\ast \varphi, (S\ast \varphi)\ast \Phi\rangle.$$
		Let us compute the convolution on the right side of the product. Let $h\in \Gamma\bs G.$
		\begin{align*}
			(S\ast \varphi)\ast \Phi (h)=&\int_G \int_{\tilde S} \varphi(s^{-1}hg^{-1})d\mathcal{H}^k(sK) \Phi(g)dg\\
			=& \int_G\int_{\tilde S} \varphi(s^{-1}hg^{-1})\Phi(g)d\mathcal H^k(sK) dg\\
			=&\int_{\tilde S}(\varphi\ast \Phi)(s^{-1}h)d\mathcal H^k(sK)\\
			=&\int_{\tilde S}(\Phi\ast \varphi)(s^{-1}h)d\mathcal H^k(sK)\\
			=&\int_G\int_{\tilde S} \Phi(s^{-1}hg^{-1})\varphi(g)d\mathcal H^k(sK)dg\\
		\end{align*}
		The second to last equality above is where we use the assumption that $\Phi$ is spherical, hence $\varphi\ast \Phi=\Phi\ast \varphi$. We can now evaluate the inner product in question
		\begin{align*}
			\langle S\ast \varphi, S\ast \varphi\ast\Phi \rangle=& \int_{\Gamma\bs G}(S\ast \varphi)(h)\int_G\int_{\tilde S} \Phi(s^{-1}hg^{-1})\varphi(g)d\mathcal H^k(sK)dgdh\\
			=& \int_{\Gamma\bs G}\int_{\tilde S}\varphi(u^{-1}h)d\mathcal H^k(uK)\int_G\int_{\tilde S} \Phi(s^{-1}hg^{-1})\varphi(g)d\mathcal H^k(sK)dg dh\\
			=& \int_{\Gamma\bs G}\left(\sum_{\gamma\in\Gamma}\int_{S_0}\varphi(u^{-1}\gamma h)d\mathcal H^k(uK)\right)\\ &\left(\int_G\int_{\tilde S}\Phi(s^{-1}hg^{-1})\varphi(g)d\mathcal H^k(sK)dg\right) dh\\
			=& \int_{\Gamma\bs G}\left(\sum_{\gamma\in \Gamma}\int_{S_0}\int_G\int_{\tilde S}\Phi(s^{-1}\gamma hg^{-1})\varphi(g)\varphi(u^{-1}\gamma h) d\mathcal H^k(sK)dg d\mathcal H^k(uK)\right) dh\\
			=& \int_{S_0}\int_G\int_G\int_{\tilde S}\Phi(s^{-1}hg^{-1})\varphi(g)\varphi(u^{-1}h)d\mathcal H^k(sK)dhdg d\mathcal H^k(uK)\\
			=& \int_{S_0}\int_G\int_G\int_{\tilde S}\Phi(s^{-1}uh)\varphi(g)\varphi(hg)d\mathcal H^k(sK)dhdg d\mathcal H^k(uK)\\
			=& \int_{S_0}\int_G \left(\int_{\tilde S}\Phi(s^{-1}uh)(\varphi\ast \check{\varphi})(h)d\mathcal H^k(sK)\right)dh d\mathcal H^k(uK).
		\end{align*}
		
		To go from the third-last to the second-last line, we do change of variables on $h$ first by left multiplication by $u$ and then by right multiplication by $g$.   
		
	\end{proof}
	\subsection{Monotonicity and the decay of matrix coefficients}\label{sec-stationaryproof}
	In this section we restrict to the rank one case. Let $\ff$ be the function defined in (\ref{defn-ffdef}). Put
	\begin{align*}
		v_S(g,r):=& \int_{\tilde S} \chi_r(\ff(s^{-1}g))d\mathcal H^k(sK)\quad \text{ (same as before)},\\
		w_S(g,r):=& \int_G v_S(gh,r)(\varphi\ast \check\varphi)(h)dh= \int_G \left(\int_{\tilde S}\chi_r(\ff(s^{-1}gh))(\varphi\ast \check{\varphi})(h)d\mathcal H^k(sK)\right)dh.
	\end{align*}
	
	Applying Lemma \ref{lem-BallIntForm} to $\Phi(g)=\chi_r(\ff(g))$, we get the identity 
	\begin{equation}\label{eq-PsiIntegralformula}\int_G \chi_r(\ff(g))\Psi(g)dg=\int_{S_0}\int_G v_S(sh,r)(\varphi\ast\check{\varphi})(h)dh d\mathcal H^k(sK)=\int_{S_0}w_S(s,r)d\mathcal H^k(sK).
	\end{equation}
	Recall $\mathfrak u:=S\ast \varphi\in L^2(\Gamma\bs G)$. Put $\mathfrak u_0:=\mathfrak u-\frac{\mathcal H^k(S)}{\vol(\Gamma\bs G)}$. Then, since $\int_G\varphi(g)dg=1$, we have $\int_{\Gamma\bs G}\mathfrak u_0(g)dg=0.$ Put $\mathfrak u_1:=\frac{\mathcal H^k(S)}{\vol(\Gamma\bs G)}.$ Let 
	$$\Psi_i(g):=\langle \mathfrak u_i,\lambda_{\Gamma\bs G}(g^{-1})\mathfrak u_i\rangle, i=0,1.$$
	Then $\Psi(g)=\langle \mathfrak u,\lambda_{\Gamma\bs G}(g^{-1})\mathfrak u\rangle= \Psi_0(g)+\Psi_1(g)$. It is not hard to compute $\Psi_1(g)=\frac{\mathcal H^k(S)^2}{\vol(\Gamma\bs G)},$ so 
	\begin{equation}\label{eq-PsiDecomp}
		\Psi(g)=\Psi_0(g) + \frac{\mathcal H^k(S)^2}{\vol(\Gamma\bs G)}.
	\end{equation}
	The function $\Psi$ plays a key role in the proof of Theorem \ref{mt-stationary}. 
	
	\begin{lem}\label{lem-matrixmono} Suppose $S\subset \Gamma\bs X$ is a nonzero stationary $k$-varifold. Then
		\begin{enumerate}
			\item $\Psi(1)\geq c \mathcal H^k(S)$,
			\item $\int_G \chi_0(\ff(g))\Psi(g)dg\geq c \Psi(1)$,
			\item  for any $r\geq 0$  $$\int_{G}\chi_r(\ff(g))\Psi(g)dg\geq c e^{\kappa(k)\|\alpha\|r}\Psi(1).$$
		\end{enumerate}
	\end{lem}
	\begin{proof}
		\begin{enumerate}
			\item $$\Psi(1)=\langle S\ast\varphi,S\ast\varphi\rangle=\int_{S_0}\int_{\tilde S}(\varphi\ast\check\varphi)(s^{-1}u)d\mathcal H^k(uK)d\mathcal H^k(sK),$$
			Let $\varepsilon>0$ be such that $\varphi\ast\check\varphi(g)\geq \varepsilon$ for all $g\in B_G(\varepsilon)$. By Lemma \ref{lem-smallballmono}, $\mathcal H^k(s^{-1} \tilde S\cap B(\varepsilon))>\delta=\delta(\varepsilon)>0$ for $\mathcal H^k$-almost all $s\in \tilde S$. Therefore $\Psi(1)\geq \varepsilon \delta \mathcal H^k(S).$ The constants $\varepsilon, \delta$ are independent of $S$ and $\Gamma$. 
			\item 
			Since $\varphi, \chi(\ff)$ are both non-negative continuous and positive at $1$, there is a constant $c>0$ such that $\varphi\ast \chi_0(\ff)\geq c \varphi.$
			We have 
			$$\int_G \chi_0(\ff(g))\Psi(g)dg=\langle S\ast \varphi, S\ast\varphi\ast \chi_0(\ff)\rangle\geq c\langle S\ast \varphi, S\ast\varphi\rangle=c \Psi(1).$$
			\item By (\ref{eq-PsiIntegralformula}), Lemma \ref{lem-VarifoldMono} and point (2) proved just above
			\begin{align*}\int_{G}\chi_r(\ff(g))\Psi(g)dg=&\int_{S_0}\int_G v_S(uhK,r)(\varphi\ast\check\varphi)(h)dh d\mathcal H^k(uK) \\
				\geq& e^{\kappa(k)\|\alpha\|r}\int_{S_0}\int_G v_S(uhK,0)(\varphi\ast\check\varphi)(h)dh d\mathcal H^k(uK)\\
				=& e^{\kappa(k)\|\alpha\|r}\int_G \chi_0(\ff(g))\Psi(g)dg\geq c e^{\kappa(k)\|\alpha\|r}\Psi(1).
			\end{align*}
		\end{enumerate}
	\end{proof}
	We have now enough tools to prove Theorem \ref{mt-stationary}. We reproduce the statement below. 
	\begin{reptheorem}{mt-stationary}
		Let $M$ be a compact octonionic hyperbolic $16$-manifold. There is a constant $c>0$ such that any nonempty stationary integral $k$-varifold $S\subset M$ satisfies $\mathcal{H}^{k}(S)\geq c \vol(M)$ for $k=14,15$.  
	\end{reptheorem}
	\begin{proof}[Proof of Theorem \ref{mt-stationary}]
		Any compact octonionic manifold is of the form $M=\Gamma\bs X$ where $\Gamma$ is a lattice in $F_4^{(-20)}$ and $X=\mathbb H_\mathbb O^2$. In that case $\vol(M)=\vol(\Gamma\bs G).$
		By Lemma \ref{lem-matrixmono}, for any $r\geq 0$
		$$\int_G\chi_r(\ff(g))\Psi(g)dg\geq c e^{\kappa(k)\|\alpha\|r}\Psi(1).$$
		On the other hand 
		\begin{align*}
			\int_G\chi_r(\ff(g))\Psi(g)dg=&\int_G\chi_r(\ff(g))\Psi_1(g)dg+ \int_G\chi_r(\ff(g))\Psi_0(g)dg\\
			=&\frac{\mathcal H^k(S)^2}{\vol(\Gamma\bs G)}\int_G\chi_r(\ff(g))dg+ \int_G\chi_r(\ff(g))\Psi_0(g)dg.
		\end{align*}
		Therefore 
		\begin{align}\label{eq-VolumeMonoDecay}
			\frac{\mathcal H^k(S)^2}{\vol(\Gamma\bs G)}\int_G\chi_r(\ff(g))dg\geq c e^{\kappa(k)\|\alpha\|r}\Psi(1) - \int_G\chi_r(\ff(g))\Psi_0(g)dg.
		\end{align}
		After unfolding the definition of $\ff$ and $\chi_r$, and integrating in spherical coordinates (Equation \ref{eq-sphericalcoord}) and changing variable $t:=\alpha(H)$ we get:
		\begin{align*}\int_G\chi_r(\ff(g))dg \leq& c \int_0^{(r+1)\|\alpha\|} \sinh(t)^{m_\alpha}\sinh(2t)^{m_{2\alpha}}dt\leq c e^{r(m_\alpha+2m_{2\alpha})\|\alpha\|}.
		\end{align*}
		The function $\Psi_0$ is a spherical matrix coefficient of the unitary representation $$L^2_0(\Gamma\bs G):=\left\{f\in L^2(\Gamma\bs G)\mid \int_{\Gamma\bs G}fdg=0\right\}.$$ The latter has no fixed vectors, so after decomposing this representation into irreducible unitary representations Lemma \ref{lem-rankonedecay} yields 
		$$|\Psi_0(k_1e^Hk_2)|\leq c e^{(1-m_{2\alpha})\alpha(H)}(1+\|H\|)\Psi_0(1)\leq  c e^{(1-m_{2\alpha})\alpha(H)}(1+\|H\|)\Psi(1).$$
		Using spherical integration once again, we get
		\begin{align*}\left|\int_G\chi_r(\ff(g))\Psi_0(g) dg \right|
			\leq& c \Psi(1)\int_0^{(r+1)\|\alpha\|} \sinh(t)^{m_\alpha}\sinh(2t)^{m_{2\alpha}}e^{(1-m_{2\alpha})t}(1+t)dt\\
			\leq& c \Psi(1) e^{r(m_\alpha+m_{2\alpha}+1)\|\alpha\|}(1+r).
		\end{align*}
		Plugging the previous inequalities back to (\ref{eq-VolumeMonoDecay}) we get 
		\begin{align*} \frac{\mathcal H^k(S)^2}{\vol(\Gamma\bs G)}e^{(m_\alpha+2m_{2\alpha})\|\alpha\|r}\geq& \Psi(1)(c_1 e^{\kappa(k)\|\alpha\|r}-c_2 e^{(m_\alpha+m_{2\alpha}+1)\|\alpha\|r}(r+1))\\
			\frac{\mathcal H^k(S)^2}{\vol(\Gamma\bs G)}\geq& \Psi(1)(c_1 e^{(\kappa(k)-m_\alpha-2m_{2\alpha})\|\alpha\|r}-c_2 e^{(1-m_{2\alpha})\|\alpha\|r}(r+1))\\
		\end{align*}
		Put $A(r):=(c_1 e^{(\kappa(k)-m_\alpha-2m_{2\alpha})\|\alpha\|r}-c_2 e^{(1-m_{2\alpha})\|\alpha\|r}(r+1))$. Our conditions on $k$ (i.e. $k=14,15$) are chosen exactly to guarantee that $\kappa(k)-m_\alpha-2m_{2\alpha}> 1-m_{2\alpha}$ (i.e. $\kappa(k)> 1+m_\alpha+m_{2\alpha}=\dim X$). We choose big enough $r$ so that $A(r)>0$. Then, by Lemma \ref{lem-matrixmono},
		\begin{align*}\frac{\mathcal H^k(S)^2}{\vol(\Gamma\bs G)}\geq & \Psi(1)A(r)\geq c \mathcal H^k(S) A(r)\\
			\mathcal H^k(S)\geq&  c A(r)\vol(\Gamma\bs G).\\
		\end{align*}
		Theorem \ref{mt-stationary} is proved.
	\end{proof}

\subsection{Proof of Theorem \ref{mt-stationary}: Higher Rank Case} 

An argument similar to the previous argument for octonionic hyperbolic manifolds will work in the $SL(n,\mathbb{R})$-case, but with the higher rank monotonicity formula (Lemma \ref{lem-VarifoldMonohigherrank}) replacing the rank one monotonicity formula, and the theorem by Oh on decay of matrix coefficients replacing Kostant's theorem.


  We lift the function $F_{\epsilon}$ from Section \ref{sec-MonoHigherRank} to a $K$-invariant function on $G$, and by abuse of notation also denote it by $F_{\epsilon}$. 
  

\vspace{2mm}
    
		Any compact Riemannian manifold with universal cover the symmetric space for $\SL(n,\mathbb{R})$ is of the form $M=\Gamma\bs X$ where $\Gamma$ is a lattice in  the isometry group of $X$ and $X=SL(n,\mathbb{R})/SO(n,\mathbb{R})$. We will assume for simplicity in what follows that $\Gamma$ is contained in $PSL(n,\mathbb{R})$, the identity component of the isometry group of $X$. The proof in the general case has a similar proof.  Note that $\vol(M)/\vol(\Gamma\bs G)$ is equal to a universal constant.  

		Lemma \ref{lem-matrixmono} also holds verbatim for $X$ the $\SL(n,\mathbb{R})$ symmetric space and with $\kappa(k)||\alpha||$ replaced by $\frac{2\langle \rho, \lambda \rangle - d (n-1-\frac{t}{2}) - \eta}{|\lambda|}$ for $d$ the codimension of the stationary varifold $S$, where the exponent comes from Lemma \ref{lem-VarifoldMonohigherrank} instead of Lemma \ref{lem-VarifoldMono}. We therefore have that for any $r\geq 0$, $\eta>0$, and $\epsilon=\epsilon(r,\eta)$ small enough
		$$\int_G\chi_r(F_{\epsilon}(g)/||\lambda||)\Psi(g)dg\geq c e^{\left(\frac{2\langle \rho, \lambda \rangle - d (n-1-\frac{t}{2}) - \eta}{|\lambda|}\right)  r}\Psi(1).$$

   \noindent      On the other hand 

\[
\int_G\chi_r(F_{\epsilon}(g)/||\lambda||)\Psi(g)dg =\frac{\mathcal H^k(S)^2}{\vol(\Gamma\bs G)}\int_G\chi_r(F_{\epsilon}(g)/||\lambda||)dg + \int_G\chi_r(F_{\epsilon}(g)/||\lambda||)\Psi_0(g)dg, 
\]
        
\noindent where as above $\Psi_0$ is the mean-zero part of $\Psi$. Therefore: 
		\begin{align}\label{eq-VolumeMonoDecayslnr}
			\frac{\mathcal H^k(S)^2}{\vol(\Gamma\bs G)}\int_G\chi_r(F_{\epsilon}(g)/||\lambda||)dg\geq c e^{\frac{2\langle \lambda, \rho \rangle - d (n-1-\frac{t}{2}) - \eta}{|\lambda|}  r}\Psi(1) - \int_G\chi_r(F_{\epsilon}(g)/||\lambda||)\Psi_0(g)dg.
		\end{align}

Let
\[
\alpha_i=e_i-e_{i+1},\qquad 1\leq i\leq n-1,
\]
be the simple roots. Recall that
\[
2\rho=\sum_{1\leq p<q\leq n}(e_p-e_q)
\]
and that, for the maximal strongly orthogonal system used in Theorem \ref{thm-decayslnr},
\[
2\theta=\sum_{j=1}^{\lfloor n/2\rfloor}(e_j-e_{n+1-j}).
\]
Since the positive root $e_p-e_q$ contains $\alpha_i$ in its simple root expansion exactly when $p\leq i<q$, while $e_j-e_{n+1-j}$ contains $\alpha_i$ exactly when $j\leq \min\{i,n-i\}$, we have
\begin{equation}\label{eq:rho-theta-simple-root-expansions}
2\rho=\sum_{i=1}^{n-1}i(n-i)\alpha_i,
\qquad
2\theta=\sum_{i=1}^{n-1}\min\{i,n-i\}\alpha_i.
\end{equation}
For every $1\leq i\leq n-1$,
\[
i(n-i)\geq \lceil n/2 \rceil\min\{i,n-i\}.
\]


\noindent We therefore have that 

 \begin{equation}\label{eq:theta-rho-over-n}
 \theta(H)\leq \frac{1}{\lceil n/2 \rceil }\rho(H)
 \qquad\text{for every }H\in\mathfrak a^+.
 \end{equation}

\noindent Therefore 
\begin{equation} \label{ineq:lambdarho}
\lambda(H) \geq \left( 1- \frac{t}{2\lceil n/2 \rceil} \right) \rho (H)  \hspace{3mm} \implies \hspace{3mm} 2 \rho (H) \leq \left( \frac{ 4\lceil n/2 \rceil } {2 \lceil n/2 \rceil - t} \right) \lambda(H)
\end{equation}

We will also use the next estimate for integration over the positive Weyl chamber, which follows by integrating in polar coordinates. 
For every $a>0$ and every integer $b\geq 0$ there is a constant $C_{n,a,b}$ such that
\begin{equation}\label{eq:cone-integral-slnr}
\int_{\substack{H\in\mathfrak a^+\\ \lambda(H)\leq T}}
e^{a\lambda(H)}(1+\|H\|)^b\,dH
\leq
C_{n,a,b}(1+T)^{n-2+b}e^{aT}.
\end{equation}
This is because the level sets of $\lambda$ scale with degree $n-2$, since $\dim\mathfrak a=n-1$, and $\|H\|$ is bounded on $\mathfrak a^+$ by a constant multiple, depending on $n$ and $t$, of $\lambda(H)$.

For $H\in\mathfrak a^+$, using that the identity element of the Weyl group occurs in the summation in the formula for $f_{\epsilon}$, we get
\begin{equation}\label{eq:fepsilon-dominates-rho}
f_{\epsilon}(H)
=\epsilon\log\left(\sum_{\omega\in W}e^{\langle\omega\lambda,H\rangle/\epsilon}\right)
\geq \lambda(H).
\end{equation}
It follows that
\[
\chi_r(f_{\epsilon}(H)/\|\lambda\|)\neq 0
\quad\Longrightarrow\quad
\lambda(H)\leq (r+1)\|\lambda\|.
\]
Moreover, the Cartan Jacobian satisfies
\[
\prod_{\alpha\in\Phi^+}\sinh(\alpha(H))^{m_\alpha}
\leq c e^{2\rho(H)}.
\]
Using Cartan integration, \eqref{eq:cone-integral-slnr}, and \eqref{ineq:lambdarho} we therefore obtain, for $r\geq 1$,
\begin{equation}\label{firstdelta}
\begin{aligned}
\int_G\chi_r(F_{\epsilon}(g)/\|\lambda\|)\,dg
&\leq
c\int_{\mathfrak a^+}\chi_r(f_{\epsilon}(H)/\|\lambda\|)
\left(\prod_{\alpha\in\Phi^+}\sinh(\alpha(H))^{m_\alpha}\right)dH\\
&\leq   c\int_{\mathfrak a^+}\chi_r(f_{\epsilon}(H)/\|\lambda\|)
e^{2\rho(H)} dH  \leq  c\int_{\mathfrak a^+}\chi_r(f_{\epsilon}(H)/\|\lambda\|)
e^{\left( \frac{ 4\lceil n/2 \rceil } {2 \lceil n/2 \rceil - t} \lambda(H) \right) } dH\\
&\leq c(1+r)^{n-2}e^{\left(\frac{ 4\lceil n/2 \rceil } {2 \lceil n/2 \rceil - t} \right) |\lambda| r}.
\end{aligned}
\end{equation}

\noindent The function $\Psi_0$ is a spherical matrix coefficient of the unitary representation
\[
L^2_0(\Gamma\bs G)
:=
\left\{f\in L^2(\Gamma\bs G):\int_{\Gamma\bs G}f\,dg=0\right\}.
\]
The latter has no fixed vectors, so after decomposing this representation into irreducible unitary representations, taking $\delta=1-t$ in Theorem \ref{thm-decayslnr} gives
\[
|\Psi_0(k_1e^Hk_2)|
\leq
c_\delta e^{-t\theta(H)}(1+\|H\|)\Psi(1).
\]
Using Cartan integration, \eqref{eq:theta-rho-over-n}, the estimate for the Jacobian, and recalling that we set $1-\delta=t$ we get
\begin{align}
\left|\int_G\chi_r(F_{\epsilon}(g)/\|\lambda\|)\Psi_0(g)\,dg\right|
&\leq
c_\delta\Psi(1)
\int_{\mathfrak a^+}
\chi_r(f_{\epsilon}(H)/\|\lambda\|)\notag\\
&\qquad\times
e^{2\rho(H)-t\theta(H)}(1+\|H\|)\,dH\notag\\
&\leq
c_\delta\Psi(1)
\int_{\substack{H\in\mathfrak a^+\\ \lambda(H)\leq(r+1)\|\lambda\|}}
e^{2\lambda(H)}(1+\|H\|)\,dH\notag\\
&\leq
c_\delta\Psi(1)(1+r)^{n-1} e^{2|\lambda|r}.
\label{ineq:meanzero}
\end{align}


\noindent Substituting the previous inequality into (\ref{eq-VolumeMonoDecayslnr}) we get 

\begin{equation} \label{ineq:1886}
\frac{\mathcal H^k(S)^2}{\vol(\Gamma\bs G)}\int_G\chi_r(F_{\epsilon}(g)/||\lambda||)dg
\geq c e^{\frac{2\langle \lambda, \rho \rangle - d (n-1-\frac{t}{2}) - \eta}{|\lambda|}  r}\Psi(1) - c (1+r)^{n-1} e^{2|\lambda|r} \Psi(1)
\end{equation}

 \noindent Recalling the upper bound \eqref{firstdelta} for the integral on the left side of the inequality directly above, to finish the argument, we need 
\[
\frac{2\langle \lambda, \rho \rangle - d (n-1-\frac{t}{2}) - \eta}{|\lambda|} > 2|\lambda|, 
\]
which recalling that $\lambda= \rho - \frac{t}{2} \theta$ is the same as 
\begin{equation} \label{eq-1714}
d(n-1-t/2) + \eta < t \langle \rho, \theta \rangle - \frac{t^2}{2} |\theta|^2. 
\end{equation}

For $SL(n,\mathbb R)$, the coordinate formulas for $\rho$
and $\theta$ give
\[
\langle\rho,\theta\rangle-\frac12\|\theta\|^2
=\frac{n(n-1)}8.
\]
Thus, if $d\le n/8$, then if $t$ were equal to $1$ and $\eta$ were equal to zero we would have that the difference of the right side of (\ref{eq-1714}) with the left side satisfies
\[
\langle\rho,\theta\rangle-\frac12\|\theta\|^2
-d\left(n-1-\frac12\right)
\ge
\frac{n(n-1)}8-\frac n8\left(n-\frac32\right)
=\frac n{16}>0.
\]
By continuity, we may choose $t<1$ sufficiently close to
$1$, and then $\eta>0$ sufficiently small, so that
\[
d(n-1-t/2)+\eta
<
t\langle\rho,\theta\rangle-\frac{t^2}{2}\|\theta\|^2.
\]
Recall that $\delta=1-t>0$, and fix $\epsilon>0$ small enough
for the uniform Hessian and monotonicity estimates to hold.
These choices depend only on $n$ and $d$.

Since the preceding inequality is equivalent to
\[
\frac{2\langle\rho,\lambda\rangle
-d(n-1-t/2)-\eta}{\|\lambda\|}
>2\|\lambda\|,
\]
the monotonicity term grows faster than the mean-zero
term, including its polynomial factor. We may therefore
choose $r$ large enough that
\[
c_\delta(1+r)^{n-1}e^{2\|\lambda\|r}
\le
\frac c2
e^{\left(
\frac{2\langle\rho,\lambda\rangle
-d(n-1-t/2)-\eta}{\|\lambda\|} \right) r }.
\]
Substituting into \eqref{eq-VolumeMonoDecayslnr} gives
\[
\frac{\mathcal H^k(S)^2}{\vol(\Gamma\bs G)}
\int_G
\chi_r\!\left(\frac{F_\epsilon(g)}{\|\lambda\|}\right)\,dg
\ge
\frac c2
\exp\!\left(
\frac{2\langle\rho,\lambda\rangle
-d(n-1-t/2)-\eta}{\|\lambda\|}\,r
\right)\Psi(1).
\]

For this fixed choice of $r$, the upper bound
\eqref{firstdelta} bounds the integral on the left by a constant independent of $\Gamma$ and $S$. Dividing
by this upper bound and absorbing the fixed exponential
factor into the constant, we obtain
\[
\frac{\mathcal H^k(S)^2}{\vol(\Gamma\bs G)}
\ge c\,\Psi(1)
\ge c'\mathcal H^k(S),
\]
where the last inequality follows from Lemma~6.10.
Dividing by $\mathcal H^k(S)>0$, we conclude that
\[
\mathcal H^k(S)\ge c'\vol(\Gamma\bs G).
\]
All parameters and constants were chosen independently
of $\Gamma$ and $S$. This proves the claim for every
integer $1\le d\le\lfloor n/8\rfloor$.

\subsection{Modifications for general split simple real higher rank groups $G$ }

We describe now the modifications for the general case.  As described at the end of section 5, we obtain a monotonicity formula with the quantity
\[
2\|\rho\|
- \frac{ d \max _{\alpha \in \Phi^+} \langle \rho, \alpha \rangle + \eta}{|\rho|}
\]
for the exponent of growth of the codimension-d stationary varifold, in place of the quantity 
\[
\frac{2\langle\rho,\lambda\rangle
-d(n-1-t/2)-\eta}{\|\lambda\|}
\]
that appeared in the argument above.

 We can also use Theorem \ref{thm-decayslnr} to control the rate of decay of matrix coefficients, where the rate of decay is given by a maximal orthogonal system $\theta$ for the group $G$. In place of (\ref{ineq:1886}), we would instead have: 
\begin{equation} \label{ineq:2226}
 \frac{\mathcal H^k(S)^2}{\vol(\Gamma\bs G)}\int_G\chi_r(F_{\epsilon}(g)/||\rho||)dg
\geq c e^{2\|\rho\|
- \left( \frac{d \max _{\alpha \in \Phi^+} \langle \rho, \alpha \rangle + \eta}{||\rho||}\right) r}\Psi(1) - c (1+r)^{\rank(\text{G})-1} e^{(\left(2 - \frac{t}{3 \rank(G)}\right)||\rho||r} \Psi(1), 
\end{equation}
\noindent where $t$ can be taken as close to $1$ and $\eta$ as close to $0$ as desired.  To obtain the second term on the right side we are using that $\theta(H) \geq \frac{1}{3\rank(\text{G})} \rho (H)$ for root systems of all types to bound $2\rho - t \theta$ in the step analogous to (\ref{ineq:meanzero}).  This can be checked using Section \ref{sec-decayhigherrank}.    

From (\ref{ineq:2226}), one can also verify using the information on the root systems from Section \ref{sec-decayhigherrank} and table \ref{tab:weyl_vector_norm} that there is a uniform constant $c>0$ so that for $d  \leq  c \left( \rank(\text{G})\right)$  the exponent in the first term of the right side of (\ref{ineq:2226}) is greater than the exponent in the second term of the right side of  (\ref{ineq:2226}) provided $t$ was chosen sufficiently close to 1.  



From here the argument can proceed as in the $SL(n,\mathbb{R})$ case above.  Finally, we can choose $c$ small enough that the statement of the theorem is vacuously true for the group of exceptional type.


        \appendix{} \label{sectiononsweepouts}
\section{Sweepouts}\label{sec-sweepouts}

In this section we collect some facts about sweepouts that are used in the main body of the paper.  First we define the space of mod-2 Lipschitz cycles in the flat norm.  Our presentation closely follows \cite{g09}. For a compact Riemannian manifold $M$, a mod-2 Lipschitz k-chain in $M$ is a sum $\sum_{i=1}^N a_i f_i$ where $a_i \in \mathbb{F}_2$ and each $f_i$ is a Lipschitz map from the standard k-simplex to $M$.  The space of all mod-2 Lipschitz k-chains is naturally a vector space over $\mathbb{F}_2$.  The boundary of a mod-2 Lipschitz chain is defined as in singular homology.  We denote the set of  mod-2 Lipschitz k-chains that lie in the kernel of the boundary map-- the space of mod-2 Lipschitz \textit{k-cycles}-- by $Z^*(M,k)$.  Since Lipschitz functions are differentiable almost everywhere, one can define the volume of an element of $Z^*(M,k)$ in the natural way. 

Let $Z_0^*(M,k)$ be the subset of $Z^*(M,k)$ consisting of cycles that are trivial as elements of singular homology.  We define the flat metric $d_{\mathbb{F}}$ on $Z_0^*(M,k)$ as follows.  For $a,b \in Z_0^*(M,k)$, let $\mathcal{C}$ be the collection of all pairs $(C_1,C_2)$ of  mod-2 Lipschitz $k$-chains and $k+1$-chains $C_1$ and $C_2$ so that $a-b= C_1 + \partial C_2$, and set 
\[
d_{\mathbb{F}}(a,b):= \inf_{(C_1,C_2)\in \mathcal{C}} \vol _k (C_1) + \vol_{k+1}(C_2).  
\]
\noindent The flat distance between two different cycles $a$ and $b$ can be zero, if for example one of them contains $k$-cycles of zero volume.  We quotient $Z_0^*(M,k)$ by the equivalence relation of being at flat distance zero to obtain a metric space whose completion we denote by $Z_0(M,k)$.  We define the volume of an element $c$ of  $Z_0(M,k)$ to be 
\[
\inf \lim_{m \to \infty } \inf \vol_k(c_m),
\]
\noindent where the infimum is taken over all sequences $\{c_m\}$ in $Z_0^*(M,k)$ converging to $c$.

A k-sweepout $F:X \rightarrow Z_0(M,k)$ is a map from a polyhedral complex $X$ that is continuous in the flat topology.  For each sweepout $F$ we define the volume $V(F)$ of $F$ to be $\sup_{x\in X} \vol(F(x))$.  Almgren showed that $H^{n-k}(Z_0(M,k);\mathbb{F}_2)$ contains a non-zero element, which we call $a(M,k)$, that corresponds to the fundamental class of $M$ \cite{almgren1962homotopy}.  We say that a sweepout $F: X \rightarrow Z_0(M,k)$ is nontrivial if it detects $a(M,k)$ in the sense that $F^*(a(M,k))$ is nonzero in $H^{n-k}(X;\mathbb{F}_2)$.  We define the \textit{k-waist} of $M$ to be the infimal volume over all nontrivial k-sweepouts with \textit{no concentration of mass}:  
\begin{equation} \label{noconcentrationofmass}
\lim_{r\downarrow 0}
\sup_{\substack{x\in X\\ p\in M}}
\|F(x)\|(B(p,r))=0
\end{equation}

The following theorem due to Almgren will be important for us.  Later Gromov gave a shorter proof of a similar statement \cite[pg. 134]{gromov1983filling} (see also \cite[Proposition 2.1]{g09}.)   

\begin{thm} \label{positivewaist}
	Let $M$ be a compact Riemannian manifold of dimension $n$.  Then the k-waist of $M$ for $0<k<n$ is positive.  
	\end{thm}


\subsection{} 
We now describe a useful criterion for a sweepout to be nontrivial.  To do this we first need to introduce the notion of a complex of cycles, which will serve as a discrete approximation to a sweepout.  A complex of cycles $\mathcal{C}$ is parametrized by a polyhedral complex $X$.  It associates to each $i$-face $\Delta$ of $X$ an $i+k$ chain $\mathcal{C}(\Delta)$ in $M$.  These chains must satisfy the following compatibility condition: if $A$ is an $i$-face in $X$ with $\partial A= \sum B_\ell$, then $\partial \mathcal{C}(A)= \sum \mathcal{C}(B_{\ell})$.  Note that this implies that vertices of $X$ are sent to cycles.    This also implies that a sum of $i$-dimensional faces in $X$ that adds up to a cycle gives an $i+k$ cycle in $M$, and so we get a map $H_i(X;\mathbb{F}_2) \rightarrow H_{i+k} (M;\mathbb{F}_2)$.  The case $i=n-k$ is the one that will matter for us.  

We describe how to build a complex of cycles $\mathcal{C}_F$ from a sweepout $F$.  Let $F:X \rightarrow Z_0(M,k)$  be a sweepout, and for some $\delta>0$ to be specified later choose a subdivision of the polyhedral complex structure on $X$ so that adjacent vertices $v_1$ and $v_2$ satisfy $d_{\mathbb{F}}(F(v_1)',F(v_2)')<3\delta$, for $F(v_1)'$ and $F(v_2)'$ Lipschitz k-chains in $Z_0(M,k)$ satisfying $d_{\mathbb{F}}(F(v_i),F(v_i)')<\delta$ and $\vol_k(F(v_i)') < \vol_k(F(v_i))+ \delta$.  (Recall that $Z_0(M,k)$ is the completion of the space of null-homologous Lipschitz k-chains relative to the flat metric.)  

For a vertex $v$ we define $\mathcal{C}_F(v)$ to be $F(v)'$.  For an edge $E$ joining vertices $v_1$ and $v_2$, we define $\mathcal{C}_F(E)$ to be a chain $C_E$ with $\partial C_E = F(v_1)' - F(v_2)'$ and $\vol_{k+1}(C_E) <5 \delta$, where to find $C_E$ we have used the Federer-Fleming isoperimetric inequality (see \cite[Section 3.4]{gromov1983filling})  Continuing in this way and assuming that we have defined $\mathcal{C}_F$ on all $i-1$-faces, for each i-face $A$ of $X$ we define $\mathcal{C}_F(A)$ to be a chain $C_A$ with boundary equal to 
\[
\sum_{B \in \partial A} \mathcal{C}_F(B)
\]
\noindent and $\vol_{i+k}(C_A) < C(n) \delta$, where we have again used the Federer-Fleming isoperimetric inequality.  

Having defined $\mathcal{C}_F$, we obtain a map $H_{n-k}(X;\mathbb{F}_2) \rightarrow H_{n} (M;\mathbb{F}_2)$.  Provided we choose $\delta$ small enough, this map will be well defined independent of the complex of cycles we chose \cite{g09}.  The following criterion for $F$ to be a nontrivial sweepout makes precise the sense in which $a(M,k)\in H^{n-k}(Z_0(M,k);\mathbb{F}_2)$ corresponds to the fundamental class of $M$.  

\begin{prop} \label{prop:nontrivialitycriterion}
The sweepout $F:X \rightarrow Z_0(M,k)$ is nontrivial exactly if the image of the map $H_{n-k}(X;\mathbb{F}_2) \rightarrow H_{n} (M;\mathbb{F}_2) \cong \mathbb{F}_2$ obtained from $\mathcal{C}_F$ is nontrivial.  
\end{prop}

\subsection{}

We now explain why generic generalized piecewise linear maps to lower-dimensional Euclidean spaces give sweepouts.  Let $M$ be a compact smooth  manifold of dimension $n$, and let $\Phi :M \rightarrow \mathbb{R}^d$ be a generalized piecewise linear map for $d<n$ for some triangulation of $M$.  We also require $\Phi$ to be generic, which recall means that it maps any $d+1$ vertices of a $d$-dimensional face of the triangulation of $M$ to a set of points in $\mathbb{R}^d$ that are not all contained in an affine hyperplane.  This can be arranged by an arbitrarily small perturbation of $\Phi$, by perturbing the images of the vertices and extending piecewise linearly.  Because $\Phi$ is generic, the preimage $\Phi^{-1}(p)$ of any point $p \in \mathbb{R}^d$ is a mod-2 Lipschitz $(n-d)$-cycle.  We see that it is a null-homologous cycle, and thus an element of $Z_0(M, n-d)$, by noting that the image of $\Phi$ is contained in a compact set, choosing a path joining $p$ to any point outside of this compact set, and taking the preimage under $\Phi$.  In more detail, one can choose the path $\gamma$ so that it intersects the positive-codimension faces of the triangulation of $\mathbb{R}^d$ in a finite set of points $p_1,..,p_l$. Since $\Phi$ defines a fiber bundle over the sub-interval of $\gamma$ between $p_i$ and $p_{i+1}$, it is clear that the preimages of points on this sub-interval are homologous.  One can also check that the preimage of a sub-interval $I$ of $\gamma$ containing one of the $p_i$ defines a  chain that bounds the preimages of the two endpoints of $I$.  Therefore all fibers are null-homologous.

The map $\Phi$ thus defines a continuous map $F_\Phi: \mathbb S^d \cong \mathbb{R}^d \cup \infty \rightarrow Z_0(M, n-d)$, where $F_{\Phi}$ sends the point at infinity to the empty cycle (as it does every point outside of a compact set containing the image of $\Phi$.)  We now prove Proposition \ref{prop:sweepoutappendix}.  

\begin{prop} \label{prop:sweepoutappendix} 
	For a generic generalized piecewise linear map $\Psi:M \rightarrow \mathbb{R}^d$, the map $F_{\Psi}$ is a nontrivial sweepout.  
	\end{prop}

    \begin{proof}
Choose a large cube $B \subset \mathbb{R}^d$ with $\Psi(M) \subset \int B$.  Subdivide $B$ into small (subject to a condition to be specified momentarily) polyhedral cells in general position with respect to affine images of the simplices in the triangulation of $M$, so that in particular the vertices are disjoint from $\Psi(M^{(d-1)}).$   

For a $j$-cell $A$, let $C(A)$ be the mod-2 chain obtained by taking $\Psi^{-1}(A)$ simplexwise, which has dimension $n-d+j$.  Boundary contributions along common faces of the triangulation of $M$ cancel giving 
\[
\partial C(A)
=
\sum_{\substack{A'\subset\partial A\\ \dim A'=j-1}}
C(A').
\]

At vertices, $C$ agrees with $F_{\Psi}$.  If the target cells have diameter small enough, then the $C(A)$ can be taken small enough to serve as the fillings in the complex of cycles construction above.  Since the inverse image of $\partial B$ is empty, we can take the parameter space to be the one-point compactification $S^d$ of $\mathbb{R}^d$.  The sum of the chains assigned to the d-cells is the mod-2 fundamental cycle of $M$, and Proposition \ref{prop:nontrivialitycriterion} therefore implies that $F_{\Psi}$ is non-trivial.  

    \end{proof}

The following theorem is a consequence of \cite[Theorem 4.10]{pittsbook} and the regularity theory for almost minimizing varifolds contained in \cite{pittsbook}. Pitts's theorem is stated in the setting of discrete sweepouts, but by the interpolation theorem  (see  \cite{liokumovich2018weyl}[Section 2.9]) a flat continuous family with no concentration of mass can be converted into an Almgren–Pitts (d,M)-homotopy class  to which \cite[Theorem 4.10]{pittsbook} applies, whose width is at most the maximal area of a cycle in the original sweepout.  That the sweepouts $F_\Psi$ we obtain from generic piecewise-linear maps $\Psi$ have no concentration of mass (\ref{noconcentrationofmass}) is routine to verify.


\begin{thm}
    There is a stationary integral rectifiable k-varifold in $M$ whose $k$-dimensional volume is a lower bound for the k-waist of $M$.   
\end{thm}

Proposition \ref{prop-waistsweepout} from the main body of the text, which says that the $d$-waist is bounded below by the $n-d$-dimensional Hausdorff measure of a stationary integral rectifiable varifold is a direct consequence of the previous theorem and Proposition \ref{prop:sweepoutappendix}.

\bibliography{bibliography}
\bibliographystyle{alpha}

	\end{document}